\documentclass[twoside]{article}

\usepackage{amsmath, amsfonts}
\usepackage{hyperref}
\usepackage{amscd}
\usepackage{amstext}
\usepackage{amsthm}
\usepackage{amsfonts}
\usepackage{amssymb,amsmath,makeidx,verbatim}
\usepackage[pdftex]{graphicx}
\usepackage{soul, color}
\def\pf{\noindent{\it Proof. }}

\newcommand{\B}{{\mathbb B}}
\newcommand{\C}{{\mathbb C}}

\newcommand{\R}{{\mathbb R}}
\newcommand{\Sp}{{\mathbb S}}
\newcommand{\Z}{{\mathbb Z}}
\newcommand{\X}{{\mathbb X}}
\newcommand{\tow}{\rightharpoonup}
\newcommand{\tows}{ \stackrel{*}{\tow}}

\newcounter{marnote}

\providecommand{\abs}[1]{\lvert#1\rvert}

\newtheorem{thm}{Theorem}[section]
\newtheorem{prop}[thm]{Proposition}

\theoremstyle{definition}

\theoremstyle{remark}

\begin{document}

\begin{titlepage}

\title{Twist maps as energy minimisers in homotopy classes: symmetrisation and the coarea formula}

\author{C. Morris, A.~Taheri$^\dag$}

\date{}

\end{titlepage}

\maketitle

\begin{abstract}

Let $\X = \X[a, b] = \{x: a<|x|<b\}\subset \R^n$ with $0<a<b<\infty$ fixed be an open annulus and consider the energy 
functional,
\begin{equation*}
{\mathbb F} [u; \X] = \frac{1}{2} \int_\X \frac{|\nabla u|^2}{|u|^2} \, dx, 
\end{equation*}
over the space of admissible incompressible Sobolev maps 
\begin{equation*}
{\mathcal A}_\phi(\X) = \bigg\{  u \in W^{1,2}(\X, \R^n) : \det \nabla u = 1 \text{ {\it a.e.} in $\X$ and $u|_{\partial \X} = \phi$} \bigg\},
\end{equation*}
where $\phi$ is the identity map of $\overline \X$. Motivated by the earlier works \cite{TA2, TA3} in this paper we examine 
the {\it twist} maps as extremisers of ${\mathbb F}$ over ${\mathcal A}_\phi(\X)$ and investigate their minimality properties by 
invoking the coarea formula and a symmetrisation argument. In the case $n=2$ where ${\mathcal A}_\phi(\X)$ 
is a union of infinitely many disjoint homotopy classes we establish the minimality of these extremising twists in their respective 
homotopy classes a result that then leads to the latter twists being $L^1$-local minimisers of ${\mathbb F}$ in ${\mathcal A}_\phi(\X)$. 
We discuss variants and extensions to higher dimensions as well as to related energy functionals.

\end{abstract}
\vspace{10pt}


\section{Introduction and preliminaries} \label{IntroSec}
\setcounter {equation}{0}

Let $\X = \X[a, b] = \{(x_1, ..., x_n) : a < |x| < b\}$ with $0<a<b<\infty$ fixed be an open annulus in $\R^n$ and consider the energy 
functional
\begin{equation} \label{F-energy-equation}
{\mathbb F}[u; \X] = \frac{1}{2} \int_\X \frac{|\nabla u|^2}{|u|^2 } \, dx, 
\end{equation}
over the space of incompressible Sobolev maps, 
\begin{equation} \label{admissible-first-equation}
{\mathcal A}_\phi (\X) = \bigg\{  u \in W^{1,2}(\X,\R^n) : \det \nabla u = 1 \text{ {\it a.e.} in $\X$ and $u|_{\partial \X} = \phi$} \bigg\}. 
\end{equation}
Here and in future $\phi$ denotes the identity map of $\overline \X$ and so the last condition in (\ref{admissible-first-equation}) means that 
$u \equiv x$ on $\partial \X$ in the sense of traces.

By a twist map $u$ on $\X \subset \R^n$ we mean a continuous self-map of $\overline \X$ onto itself which agrees with the identity map $\phi$ 
on the boundary $\partial \X$ and has the specific spherical polar representation ({\it see} \cite{TA3}-\cite{TA5} for background and further results) 
\begin{equation} \label{polar-twist}
u : \left( r, \theta \right) \mapsto \left(r,Q(r)\theta \right), \qquad x \in \overline \X.
\end{equation}
Here $r = |x|$ lies in $[a, b]$ and $\theta = x/|x|$ sits on ${\mathbb S}^{n-1}$ with $Q \in {\bf C}([a,b], {\bf SO}(n))$ satisfying $Q(a)=Q(b)=I$. 
Therefore $Q$ forms a closed loop in ${\bf SO}(n)$ based at $I$ and for this in sequel we refer to $Q$ as the twist {\it loop} associated with 
$u$. Also note that \eqref{polar-twist} in cartesian form can be written as
\begin{equation} \label{uQx-equation}
u: x \mapsto Q(r) x = r Q(r) \theta, \qquad x \in \overline \X.
\end{equation}
Next subject to a differentiability assumption on the twist loop $Q$ it can be verified that $u \in {\mathcal A}_\phi(\X)$ with its ${\mathbb F}$ 
energy simplifying to  
\begin{align} \label{en1}
{\mathbb F}[Q(r) x; \X] &= \frac{1}{2} \int_\X \frac{|\nabla u|^2}{|u|^2} \,dx = \frac{1}{2} \int_\X \frac{|\nabla Q(r)x|^2}{|x|^2} \, dx \nonumber \\
&= \frac{n}{2}\int_\X \frac{dx}{|x|^2} + \frac{\omega_n}{2}\int_a^b {\abs{\dot{Q}}^2}  r^{n-1} \, dr, 
\end{align} 
where the last equality uses $|\nabla [Q(r)x]|^2 = n + r^2 |\dot{Q} \theta|^2$. Now as the primary task here is to search for extremising twist 
maps we first look at the Euler-Lagrange equation associated with the loop energy ${\mathbb E}={\mathbb E}[Q]$ defined by the last integral 
in (\ref{en1}) over the loop space $\{Q \in W^{1,2}([a, b]; {\bf SO}(n)) : Q(a)=Q(b)=I\}$. Indeed this can be shown to take the form 
(see below for justification)
\begin{equation} \label{EL-twist-equation}
\frac{d}{dr} \left[ \left( r^{n-1} \dot{Q}\right)Q^t \right] =0, 
\end{equation} 
with solutions  
$Q(r) = {\rm exp} [-\beta(r) A] P$,
where $P \in {\bf SO}(n)$, $A \in \R^{n \times n}$ is skew-symmetric and $\beta=\beta(|x|)$ is described for $a \le r \le b$ by
\begin{equation}
\beta(r) = \begin{cases}
\ln 1/r & \text{$n=2$},\\
r^{2-n}/(n-2) & \text{$n\geq 3$}.
\end{cases}
\end{equation}
Now to justify (\ref{EL-twist-equation}) fix $Q\in W^{1,2}([a,b], {\bf SO}(n))$ and for $F \in W^{1,2}_0([a,b], \R^{n \times n})$ 
set $H= (F-F^t)Q$ and $Q_\epsilon = Q + \epsilon H$. Then $Q_{\epsilon}^tQ_{\epsilon} = I + \epsilon^2 H^tH$ and 
\begin{align*}
\frac{d}{d\epsilon} \int_a^b 2^{-1}{|\dot Q_\epsilon|^2} r^{n-1} \, dr \bigg|_{\epsilon = 0} 
&= \int_a^b \langle \dot{Q} , (\dot{F} - \dot{F}^t )Q + (F - F^t )\dot{Q} \rangle r^{n-1} \, dr \\
&= \int_a^b \langle \dot{Q} , (\dot{F} - \dot{F}^t )Q \rangle r^{n-1} \, dr \\
&= \int_a^b \langle \frac{d}{dr} (r^{n-1} \dot{Q}Q^t) , ({F} - {F}^t ) \rangle dr = 0,
\end{align*}
and so the arbitrariness of $F$ with an orthogonality argument gives (\ref{EL-twist-equation}).

Returning to (\ref{F-energy-equation}) it is not difficult to see that the Euler-Lagrange equation associated with ${\mathbb F}$ 
over ${\mathcal A}_\phi(\X)$ is given by the system ({\it cf.} Section \ref{EuLgSec})
\begin{equation}
\frac{\abs{\nabla u}^2}{\abs{u}^4} u + div \left\{ \frac{\nabla u}{\abs{u}^2} - p(x) {\rm cof}\,\nabla u \right\}  = 0, \qquad u=(u_1, ..., u_n), 
\end{equation}
where $p=p(x)$ is a suitable Lagrange multiplier. Here a further analysis reveals that out of the solutions $Q=Q(r)$ to 
\eqref{EL-twist-equation} just described only those twist loops in the form    
\begin{equation} \label{Qr-gr-equation}
Q(r) = R\,{\rm diag}[{\bf R}[g](r), ..., {\bf R}[g](r)]R^t,\qquad R\in\mathbf{SO}(n),
\end{equation}
when $n$ is even and $Q(r) \equiv I$ ($a \le r \le b$) when $n$ is odd can grant extremising twist maps $u$ for the original energy 
(\ref{F-energy-equation}). For clarification ${\bf R}[g]$ denotes the ${\bf SO}(2)$ matrix of rotation by angle $g$:
\begin{equation*}
{\bf R} [g] = \left[ \begin{array}{cc} 
\cos g & \sin g\\
-\sin g & \cos g\\
\end{array} \right].
\end{equation*}
Indeed direct computations give the {\it angle} of rotation $g=g(r)$ to be   
\begin{equation} \label{rotation-angle-2d-equation}
g(r) = 2\pi k \frac{\log(r/a)}{\log(b/a)} + 2\pi m, \qquad k, m \in \Z, 
\end{equation}
when $n=2$ and 
\begin{equation}
g(r) = 2\pi k \frac{(r/a)^{2-n}-1}{(b/a)^{2-n}-1} + 2\pi m, \qquad  k, m \in \Z, \label{twistsevendim}
\end{equation}
when $n \ge 4$ even. (See also \cite{ShT2}, \cite{TA3}, \cite{TA5} for complementing and further results.)

Our point of departure is (\ref{uQx-equation})-(\ref{Qr-gr-equation}) and the aim is to study the minimising properties the twist maps calculated 
above. Of particular interest is the case $n=2$ where the space ${\mathcal A}_\phi(\X)$ admits multiple homotopy classes $(A_k : k \in \Z)$. 
Here direct minimisation of the energy over these classes gives rise to a scale of associated minimisers $(u_k)$. Using a symmetrisation 
argument and the coarea formula we show that the twist maps $u_k$ with twist angle $g$ as presented in 
(\ref{rotation-angle-2d-equation}) are indeed energy minimisers in $A_k$ and as a result also $L^1$ local minimisers of ${\mathbb F}$ over 
${\mathcal A}_\phi(\X)$. We discuss variants and extensions including a larger scale of energies where similar techniques can be applied 
to establish minimality properties in homotopy classes.

\section{The homotopy structure of the space of self-maps ${\mathfrak A}={\mathfrak A}(\X)$} \label{SecTwo}
\setcounter {equation}{0}

The rich homotopy structure of the space of continuous self-maps of the annulus $\X=\X[a, b] \subset \R^n$ will prove 
useful later on in constructing local energy minimisers. For this reason here we give a quick outline of the main tools and results and refer the 
reader to \cite{TA2} for further details and proofs. To  this end set 
${\mathfrak A}={\mathfrak A} (\X)= \{ f \in {\bf C}(\overline{\X}, \overline{\X}) : f =\phi \mbox{ on $\partial \X$} \}$
equipped with the uniform topology. A pair of maps $f_0, f_1 \in {\mathfrak A}$ are homotopic {\it iff} there exists 
$H \in {\bf C}([0, 1] \times \overline \X; \overline \X)$ such that, firstly, $H(0, x) = f_0 (x)$ for all ${\bf x} \in \overline \X$, 
secondly, $H(1, x) = f_1 (x)$ for all ${\bf x} \in \overline \X$ and finally $H(t, x) = x$ for all $t \in [0,1]$, $x \in \partial \X$. 
The equivalence class consisting of all $g \in {\mathfrak A}$ homotopic to a given $f \in {\mathfrak A}$ is referred to as the homotopy 
class of $f$ and is denoted by $[f]$. Now the homotopy classes $\{[f] : f \in {\mathfrak A}\}$ can be characterised as follows depending 
on whether $n=2$ or $n \ge 3$.

\begin{itemize}

\item ($n=2$)  Using polar coordinates, for $f \in {\mathfrak A}$ and for $\theta \in [0, 2 \pi]$ ({\it fixed}), the $\Sp^1$-valued curve 
$\gamma_\theta$ defined by 
\begin{equation*}
\gamma_\theta:  [a, b] \to \Sp^1 \subset \R^2, \qquad \gamma_\theta: r \mapsto  f|f|^{-1}(r, \theta), 
\end{equation*} 
has a well-defined index or winding number about the origin. Furthermore, due to continuity of $f$, this index is independent of the 
particular choice of $\theta \in [0, 2 \pi]$. This assignment of an integer (or index) to a map $f \in {\mathfrak A}$ will be denoted by 
\begin{equation}
f \mapsto {\bf deg}(f|f|^{-1}).
\end{equation} 
Note firstly that this integer also agrees with the Brouwer {\it degree} of the map resulting from identifying $\Sp^1 \cong [a, b]/\{a, b\}$,  
justified as a result of $\gamma_\theta(a)=\gamma_\theta(b)$ and secondly that for a differentiable curve (taking advantage of the 
embedding $\Sp^1 \subset \C$) we have the explicit formulation
\begin{equation} \label{degree-2d-equation}
{\bf deg} (f|f|^{-1}) = \frac{1}{2 \pi i} \int_{\gamma} \frac{dz}{z}.
\end{equation}

\begin{prop} \label{conv-deg} $(n=2)$. The map ${\bf deg}: \{ [f] : f \in {\mathfrak A} \} \to \Z$ 
is bijective. Moreover, for any pair of maps $f_0, f_1 \in {\mathfrak A}$, we have
\begin{equation}
[f_0]=[f_1] \iff {\bf deg}(f_0 |f_0|^{-1}) = {\bf deg}(f_1|f_1|^{-1}).
\end{equation} 
\end{prop}

\item ($n \ge 3$) Using the identification $\overline \X \cong [a,b] \times \Sp^m$ where for ease of notation we have set $m=n-1$ it is plain 
that for $f \in {\mathfrak A}$ the map 
\footnote{Here as usual $\phi$ denotes the {\it identity} map of the $m$-sphere and ${\bf C}_\phi(\Sp^m, \Sp^m)$ 
is the path-connected component of ${\bf C}(\Sp^m, \Sp^m)$ containing $\phi$.}
\begin{equation*}
\omega: [a, b] \to {\bf C}_\phi(\Sp^m, \Sp^m), \qquad \omega: r \mapsto f|f|^{-1} (r, \cdot), 
\end{equation*}
uniquely defines an element of the fundamental group $\pi_1[{\bf C}_\phi(\Sp^m, \Sp^m)]$. 
By considering the action of ${\bf SO}(n)$ on $\Sp^m$ -- viewed as its group of 
orientation {\it preserving} isometries, i.e., through the assignment, 
\begin{equation}
{\bf E}: \xi \in {\bf SO}(n) \mapsto \omega \in {\bf C}(\Sp^m, \Sp^m),
\end{equation} 
where 
\begin{equation}
\omega(x) = {\bf E}[\xi](x) = \xi x, \qquad x \in {\mathbb S}^m, 
\end{equation} 
it can be proved that the latter assignment induces a group isomorphism on the level of the fundamental groups, namely, 
\begin{equation}
{\bf E}^\star: \pi_1[{\bf SO}(n), I_n] \cong  \pi_1 [{\bf C}_\phi(\Sp^m, \Sp^m), \phi] \cong \Z_2. 
\end{equation}
Thus, summarising, we are naturally lead to the assignment of an integer mod $2$ to any $f \in {\mathfrak A}$ which will 
be denoted by
\begin{equation}
f \mapsto {\bf deg}_{2} (f|f|^{-1}) \in \Z_2. 
\end{equation}

\begin{prop} $(n \ge 3)$ The degree mod $2$ map ${\bf deg}_2: \{ [f] : f \in {\mathfrak A} \} \to \Z_2$ 
is bijective. Moreover, for a pair of maps $f_0, f_1 \in {\mathfrak A}$, we have 
\begin{equation}
[f_0]=[f_1] \iff {\bf deg}_2(f_0 |f_0|^{-1}) = {\bf deg}_2(f_1|f_1|^{-1}).
\end{equation} 
\end{prop}

\end{itemize}

\section{A countable family of $L^1$ local minimisers of ${\mathbb F}$ when $n=2$}
\setcounter {equation}{0}

When $n=2$ by Lebesgue monotonicity and degree theory ({\it see} \cite{TA2} as well as \cite{MC},\cite{MST},\cite{S},\cite{VG}) 
every map $u$ in ${\mathcal A} = {\mathcal A}_\phi(\X)$ has a representative (again denoted $u$) 
in ${\mathfrak A}$. As a consequence we can introduce the components -- hereafter called the 
homotopy classes, 
\begin{equation} \label{AK-L2}
A_k := \bigg\{u \in {\mathcal A}: {\bf deg} (u|u|^{-1}) = k \bigg\}, \qquad k \in \Z.
\end{equation}
Evidently $A_k$ are pairwise disjoint and their union (over all $k \in \Z$) gives ${\mathcal A}$. Furthermore it can be seen 
without difficulty that each $A_k$ is $W^{1,2}$-sequentially weakly closed and that for $u \in A_k$ and $s>0$ there 
exists $\delta=\delta(u, s)>0$ with
\begin{equation}
\{v : {\mathbb F}[v]< s\} \cap {\mathbb B}^{L^1}_\delta(u) \subset A_k.
\end{equation}
Here ${\mathbb B}^{L^1}_\delta(u) = \{v \in {\mathcal A} : ||v-u||_{L^1} < \delta\}$, that is, the $L^1$-ball in ${\mathcal A}$ centred at $u$. 
Indeed for the sequential weak closedness fix $k$ and pick $(u_j : j \ge 1) \subset A_k$ so that $u_j \tow u$ in $W^{1,2}$. 
Then by a classical result of Y.~Reshetnyak 
\begin{equation}
\det \nabla u_j \tows \det \nabla u
\end{equation} 
(as measures) and so $u \in {\mathcal A}$ while $u_j \to u$ uniformly on $\overline \X$ gives by Proposition \ref{conv-deg} 
that $u \in A_k$. For the second assertion arguing indirectly and assuming 
the contrary there exist $u \in A_k$, $s>0$ and $(v_j : j \ge 1)$ in ${\mathcal A}$ such that ${\mathbb F}[v_j; \X] < s$ 
and $||v_j - u||_{L^1} \to 0$ 
while $v_j \notin A_k$. However by passing to a 
subsequence (not re-labeled) $v_j \tow u$ in $W^{1,2}(\X, \R^2)$ and as above 
$v_j \to u$ uniformly on $\overline \X$. Hence again by Proposition \ref{conv-deg}, $v_j \in A_k$ for large enough $j$ 
which is a contradiction. \hfill $\square$
\\[1 mm]

Now in view of the sequential weak lower semicontinuity of ${\mathbb F}$ in ${\mathcal A}$ (see below) an application 
of the direct methods of the calculus of variations leads to the following existence and multiplicity result.

\begin{thm} \label{Existence-Min} $($Local minimisers$)$ Let $\X = \X[a, b] \subset \R^2$ and for $k \in \Z$ consider the homotopy classes $A_k$ as defined by 
$(\ref{AK-L2})$. Then there exists $u= u (x; k) \in A_k$ such that
\begin{equation}
{\mathbb F}[u; \X] = \inf_{ v \in A_k} {\mathbb F}[v; \X].
\end{equation}
Furthermore for each such minimiser $u$ there exists $\delta=\delta(u)>0$ such that
\begin{equation} \label{Loc-min}
{\mathbb F}[u; \X] \le {\mathbb F}[v; \X],
\end{equation}
for all $v \in {\mathcal A}_\phi(\X)$ satisfying $||u-v||_{L^1}<\delta$.
\end{thm}

\pf Fix $k$ and pick $(v_j) \subset A_k$ an {\it infimizing} sequence: ${\mathbb F}[v_j] \downarrow L:= \inf_{A_k} {\mathbb F}[\cdot]$. Then as $L<\infty$ and 
$a \le |v(x)| \le b$ for $v \in {\mathcal A}$ it follows that by passing to a subsequence (not re-labeled) $v_j \tow u$ in $W^{1,2}(\X, \R^2)$ and uniformly in 
$\overline \X$ where by the above discussion $u \in A_k$. Now  
\begin{align*}
\left| \int_\X \frac{|\nabla v_k|^2}{|v_k|^2} - \int_\X \frac{|\nabla v_k|^2}{|u|^2} \right| &\le \int_\X |\nabla v_k|^2\,\bigg( \frac{||v_k|^2-|u|^2|}{|u|^2 |v_k|^2}\bigg)  \\
& \le \sup_{\overline \X}  \frac{||v_k|^2-|u|^2|}{|u|^2 |v_k|^2} \int_\X |\nabla v_k|^2 \to 0 
\end{align*}
as $k \nearrow \infty$ together with 
\begin{equation}
\int_\X \frac{|\nabla u|^2}{|u|^2} \le \underline{\lim} \int_\X \frac{|\nabla v_k|^2}{|u|^2} 
\end{equation}
gives the desired lower semicontinuity of the ${\mathbb F}$ energy on ${\mathcal A}_\phi(\X)$ as claimed, i.e., 
\begin{equation}
{\mathbb F}[u; \X]=\int_\X \frac{|\nabla u|^2}{2|u|^2} \le \underline{\lim} \int_\X \frac{|\nabla v_k|^2}{2|v_k|^2} = \underline{\lim} \, {\mathbb F}[v_k; \X].
\end{equation}
As a result $L \le {\mathbb F}[u] \le \liminf {\mathbb F}[v_j] \le L$ and so $u$ is a minimiser as required. 
To justify the second assertion fix $k \in \Z$ and $u$ as above and with $s=1+{\mathbb F}[u]$ pick $\delta>0$ as in the discussion prior to the theorem. 
Then any $v \in {\mathcal A}$ satisfying $||u-v||_{L^1} < \delta$ also satisfies (\ref{Loc-min}) [otherwise 
${\mathbb F}[v]<{\mathbb F}[u] < s$ implying that $v \in A_k$ and hence in view of $u$ being a minimiser, ${\mathbb F}[v] \ge {\mathbb F}[u]$ 
which is a contradiction.] \hfill $\square$

\section{Twist maps and the Euler-Lagrange equation associated with ${\mathbb F}$} \label{EuLgSec}
\setcounter {equation}{0}

The purpose of this section is to formally derive the Euler-Lagrange equation associated with ${\mathbb F}$ 
over ${\mathcal A}_\phi(\X)$. Note that ${\mathbb F}[u]$ can in principle be infinite if $|u|$ is too small or 
zero, however, for twist maps or more generally $L^n$-integrable maps in ${\mathcal A}_\phi$, $|u|$ is bounded away from zero  
as $u$ is a self-map of $\overline \X$ onto itself. Moreover $\det \nabla u$ is 
$L^1$-integrable for the latter maps but not in general for maps $u$ of Sobolev class $W^{1,2}$ (with $n \ge 3$).

Now the derivation uses the Lagrange multiplier method and proceeds formally by considering the unstrained functional
\begin{align} 
\mathbb{K}[u; \mathbb{X}] = \int_{\X} \left[ \frac{\abs{\nabla u}^2}{2\abs{u}^2} - p(x) \left(\det \nabla u -1 \right)  \right] \, dx,
\end{align} 
for suitable $p=p(x)$ where evidently ${\mathbb K}[u; \X]={\mathbb F}[u; \X]$ when $u \in {\mathcal A}_\phi(\X)$. 
We can calculate the first variation of this energy by setting 
$d/d\varepsilon \, {\mathbb K}[u_\varepsilon; \X] |_{\varepsilon=0}=0$, where $u \in {\mathcal A}_\phi(\X)$ is sufficiently 
regular and satisfies $|u| \ge c>0$ in $\X$, $u_\varepsilon=u+\varepsilon \varphi$ for all $\varphi \in {\bf C}^\infty_c(\X,\R^n)$ and 
$\varepsilon \in \R$ sufficiently small, hence obtaining,
\begin{align*}
0 =& \frac{d}{d \varepsilon} \int_{\X} \left[ \frac{\abs{\nabla u_\varepsilon}^2}{2\abs{u_\varepsilon}^2} 
- p(x) \left(\det \nabla u_\varepsilon -1 \right)  \right] \, dx \bigg|_{\varepsilon=0} \\
= &\int_{\X} \left\{ \sum_{i,j=1}^n \left[ \frac{1}{\abs{u}^2}\frac{\partial u_i}{\partial x_j} - p(x) [{\rm cof}\nabla u]_{{ij}} \right] 
\frac{\partial \varphi_i}{\partial x_j} - \sum_{i=1}^n \frac{\abs{\nabla u}^2}{\abs{u}^4}u_i \varphi_i \right\} dx  \\
= &\int_{\X} \left\{ - \sum_{i,j=1}^n \frac{\partial }{\partial x_j}\left[ \frac{1}{\abs{u}^2}\frac{\partial u_i}{\partial x_j} 
- p(x) [{\rm cof}\nabla u]_{{ij}}  \right] 
\varphi_i  - \sum_{i=1}^n  \frac{\abs{\nabla u}^2}{\abs{u}^4} u_i \varphi_i \right\} dx \\
= &\int_{\X} - \sum_{i=1}^n\left\{ \frac{\abs{\nabla u}^2}{\abs{u}^4} u_i + \sum_{j=1}^n \frac{\partial }{\partial x_j} 
\left[ \frac{1}{\abs{u}^2}\frac{\partial u_i}{\partial x_j} - p(x) [{\rm cof}\nabla u]_{{ij}} \right]    \right\} \varphi_i \, dx.
\end{align*}
As this is true for every compactly supported $\varphi$ as above an application of the fundamental lemma 
of the calculus of variation results in the Euler-Lagrange system for $u=(u_1, ..., u_n)$ in $\X$: 
\begin{align} \label{ELinitial}
\frac{\abs{\nabla u}^2}{\abs{u}^4} u + div \left\{ \frac{\nabla u}{\abs{u}^2} - p(x) {\rm cof}\,\nabla u \right\}  = 0,  
\end{align}
where the divergence operator is taken row-wise. Proceeding further an application of the Piola identity on the 
cofactor term gives (with $1 \le i \le n$)
\begin{align}
\frac{\abs{\nabla u}^2}{\abs{u}^4} u_i + \sum_{j=1}^n \left\{ \frac{\partial }{\partial x_j} 
\left( \frac{1}{\abs{u}^2}\frac{\partial u_i}{\partial x_j} \right) - [{\rm cof}\,\nabla u]_{{ij}}\frac{\partial p}{\partial x_j}   \right\} = 0.
\end{align}
Next expanding the differentiation further allows us to write
\begin{align}
0 & = \frac{\abs{\nabla u}^2}{\abs{u}^4} u_i + \sum_{j=1}^n \left\{ \frac{1}{\abs{u}^2}\frac{\partial^2 u_i }{\partial x_j^2} 
- \frac{2}{\abs{u}^4} \sum_{k=1}^n \frac{\partial u_i}{\partial x_j} \frac{\partial u_k }{\partial x_j} u_k  
- [{\rm cof}\,\nabla u]_{{ij}}\frac{\partial p}{\partial x_j}   \right\} \nonumber \\
& = \frac{\abs{\nabla u}^2}{\abs{u}^4} u_i + \sum_{j=1}^n \left\{ \frac{1}{\abs{u}^2}\frac{\partial^2 u_i }{\partial x_j^2} 
- \frac{2}{\abs{u}^4}\frac{\partial u_i}{\partial x_j} [\nabla u^tu]_{j} - [{\rm cof}\,\nabla u]_{{ij}}\frac{\partial p}{\partial x_j}   \right\}.
\end{align}

Finally transferring back into vector notation and invoking the incompressibilty condition $\det \nabla u =1$ it follows in turn that 
\begin{equation}
\frac{\Delta u}{\abs{u}^2} + \frac{\abs{\nabla u}^2}{\abs{u}^4} u - \frac{2}{\abs{u}^4}  \nabla u (\nabla u)^tu = ({\rm cof}\,\nabla u) \nabla p, 
\end{equation}
and subsequently 
\begin{equation}
\frac{(\nabla u)^t}{\abs{u}^2} \left[ {\Delta u} + \frac{\abs{\nabla u}^2}{\abs{u}^2} u - \frac{2}{\abs{u}^2}  \nabla u (\nabla u)^tu \right] = \nabla p. \label{EL}
\end{equation}
Thus the Euler-Lagrange system \eqref{ELinitial} is equivalent to \eqref{EL} that in particular asks for the nonlinear term on the left of \eqref{EL} to be a gradient field 
in $\X$. Recall from earlier discussion that restricting ${\mathbb F}$ to the class of twist maps results in the Euler-Lagrange equation (\ref{EL-twist-equation}) 
where the solution $Q=Q(r)$ as explicitly computed is the twist loop associated with the map 
\begin{align}
u: (r, \theta) \mapsto (r, Q(r) \theta), \qquad x \in \X, \label{TwistSol1}
\end{align}  
with $Q(r)=\exp [- \beta(r){A}] P$, $P \in \mathbf{SO}(n)$, $A \in \R^{n \times n}$ skew-symmetric ($A^t=-A$). 
The boundary condition $u=\phi$ on $\partial \X$ gives,
\footnote{The function $\beta=\beta(r)$ was introduced earlier in Section \ref{IntroSec}.} 
\begin{align}
\exp[-\beta(a) A] P=I, \qquad \exp[-\beta(b) A] P=I.
\end{align}
Therefore it must be that $P= \exp[\beta(a) A]$ and $\exp ([\beta(b)-\beta(a)] A)=I$. Now as $A$ lies in $\mathfrak{so}(n)$ it must be conjugate 
to a matrix $S$ in the Lie algebra of the standard maximal torus of orthogonal $2$-plane rotations in $\mathbf{SO}(n)$. This means that 
there exists $R \in \mathbf{SO}(n)$ such that $A = RSR^{T}$ for some $S$ as described and so $S\in (\beta(b)-\beta(a))^{-1} {\mathbb L}$ 
where $\mathbb{L}=\lbrace T\in \mathfrak{t}: \exp(T)=I \rbrace$, that is, 
${\mathbb L}$ is the lattice in the Lie subalgebra $\mathfrak{t} \subset {\mathfrak s}{\mathfrak o}(n)$ consisting of matrices sent by the 
exponential map to the identity $I$ of ${\bf SO}(n)$. Hence $Q(r)=R \exp (-[\beta(r)-\beta(a)] S) R^t$. 
Next upon noting that the derivatives of $\beta=\beta(r)$ are given by
\begin{align*}
\dot{\beta} (r) &= -\frac{1}{r^{n-1}} , \qquad \ddot{\beta} (r) = \frac{n-1}{r^{n}}, 
\end{align*}
we can write
\begin{align}
\dot{Q} =  \frac{A Q}{r^{n-1}}, \qquad \ddot{Q} = \frac{A^2Q}{r^{2n-2}}-(n-1)\frac{AQ}{r^{n}}.
\end{align}

Now, moving forward, a set of straightforward calculations show that for a twist map $u$ with a twice continuously differentiable twist loop $Q=Q(r)$ we have the differential relations 
\begin{align}
(\nabla u)^t &= Q^t + r \theta \otimes \dot{Q}\theta, \nonumber \\
\abs{\nabla u}^2 &= tr [(\nabla u)^t (\nabla u)] = n + r^2 \abs{\dot{Q}\theta}^2, 
\end{align}
and likewise
\begin{align}
\Delta u = \left[ (n+1)\dot Q + r \ddot{Q}  \right]\theta.  
\end{align}

Thus for the particular choice of a twist map with twist loop arising from a solution to \eqref{EL-twist-equation} the above quantities can be explicitly described by 
the relations
\begin{align}
(\nabla u)^t &= Q^t + r^{2-n} \theta \otimes AQ\theta, \label{TCY}\\
\abs{\nabla u}^2 &= tr [ (Q^t + r^{2-n} \theta \otimes AQ\theta) (Q + r^{2-n} AQ \theta \otimes \theta) ] \nonumber \\
&= n + \frac{\abs{AQ\theta}^2}{r^{2(n-2)}}, \label{TC2} 
\end{align}
and likewise 
\begin{align}
\Delta u &= \left[ (n+1) \frac{AQ}{r^{n-1}} + r\left( \frac{A^2 Q}{r^{2n-2}} -(n-1) \frac{AQ}{r^{n}} \right)   \right] \theta \nonumber \\
& = \left[ \frac{2AQ}{r^{n-1}}  + \frac{A^2Q}{r^{2n-3}}   \right]\theta.  \label{TC1}
\end{align}
For the ease of notation we shall hereafter write $\omega = Q\theta$. Proceeding now with the calculations and using $(\ref{TC2})$-$(\ref{TC1})$ we have 
\begin{align}
\Delta u + \frac{\abs{\nabla u}^2}{\abs{u}^2}u & = \left[ 2\frac{A}{r^{n-1}}  + \frac{A^2}{r^{2n-3}}  
+  \frac{1}{r}\left(n + \frac{\abs{A\omega}^2}{r^{2n-4}} \right) I \right]\omega  \label{TCX} 
\end{align}
and in a similar way 
\begin{align}
\frac{\nabla u (\nabla u)^t}{\abs{u}^2}u & = \left[  Q + r^{2-n} A\omega \otimes \theta  \right] 
\left[  Q^t + r^{2-n} \theta \otimes A\omega  \right] \frac{\omega}{r} \nonumber \\
& = \left[ I + \frac{A\omega \otimes \theta Q^t +  Q\theta \otimes A\omega}{r^{n-2}}  
+ \frac{A\omega \otimes A\omega}{r^{2n-4}}   \right] \frac{\omega}{r} \nonumber \\
& = \left[ I + \frac{A\omega \otimes \omega +  \omega \otimes A\omega}{r^{n-2}}  
+ \frac{A\omega \otimes A\omega}{r^{2n-4}}   \right] \frac{\omega}{r}. \nonumber \\ 
& =  \frac{1}{r} (I + r^{2-n} A) \omega, \label{TC4}
\end{align}
where the last identity here uses $(x\otimes y)z = \langle y,z\rangle x$ and $\langle \omega,\omega\rangle = \langle Q\theta,Q\theta\rangle=1$ along with 
$\langle A\omega,\omega\rangle =0$ for skew-symmetric $A$. Hence putting together $(\ref{TCX})$ and $(\ref{TC4})$ gives
\begin{align}
\Delta u  + \frac{\abs{\nabla u}^2}{\abs{u}^2}u - 2 \frac{\nabla u (\nabla u)^t}{\abs{u}^2}u = \left[  \frac{A^2+ \abs{A\omega}^2 I}{r^{2n-2}}  
+  (n-2) I \right] \frac{\omega}{r}, \label{TC5}
\end{align}
which when combined with $(\ref{TCY})$ results in
\begin{align}
\eqref{EL} &=\frac{(\nabla u)^t}{\abs{u}^2} \left[ \Delta u  + \frac{\abs{\nabla u}^2}{\abs{u}^2}u - 2 \frac{\nabla u (\nabla u)^t}{\abs{u}^2}u \right] \nonumber \\
&= \frac{1}{r^2} \left[  Q^t + \frac{\theta \otimes A\omega}{r^{n-2}} \right] 
\left[  \frac{A^2 + \abs{A\omega}^2 I}{r^{2n-4}}  +  (n-2) I \right] \frac{\omega}{r} \nonumber \\
&=  Q^t \left[  \frac{A^2+ \abs{A\omega}^2 I}{r^{2n-1}}  
+  \frac{n-2}{r^3} I \right]\omega + \frac{(\theta \otimes A\omega )A^2 \omega}{r^{3n-3}}. 
\end{align}
Noting $(\theta \otimes A\omega )A^2 \omega = \langle A\omega,A^2\omega \rangle \theta = \langle \mu,A\mu\rangle = 0$ 
with $A$ skew-symmetric and $\mu = A\omega$ the last set of equations give
\begin{align}
\eqref{EL} &= \frac{(\nabla u)^t}{\abs{u}^2} \left[ \Delta u  + \frac{\abs{\nabla u}^2}{\abs{u}^2}u - 2 \frac{\nabla u (\nabla u)^t}{\abs{u}^2}u \right] \nonumber \\
&= Q^t \left[  \frac{A^2+ \abs{A\omega}^2 I}{r^{2n-2}}  +  \frac{n-2}{r^2} I \right] \frac{\omega}{r} = {\bf I}. \label{fullELtwist}
\end{align}

Therefore to see if $(\ref{EL})$ admits twist solutions it suffices to verify if the quantity described by $(\ref{fullELtwist})$ is a gradient field in $\X$. Towards this end 
recall that here we have $Q(r) = \exp(-\beta(r)A)P$ where as seen $P=\exp(\beta(a)A)$. Thus a basic 
calculation gives
\begin{align*}
\abs{A\omega}^2 &= \abs{AQ\theta}^2 =  \theta^t Q^t A^t A Q \theta = - \theta^t Q^t A^2 Q \theta \\
& =  - \theta^t P^t \exp(\beta(r)A) A^2 \exp(-\beta(r)A)P \theta \\ 
&= - \theta P^t A^2 P\theta = \theta P^t A^t A P\theta \\
&= \abs{AP \theta}^2, 
\end{align*}
and likewise by substitution we have
\begin{equation*}
Q^tA^2 \omega = P^t A^2 P \theta.
\end{equation*}

Hence using the above we can proceed by writing the Euler-Lagrange equation $(\ref{fullELtwist})$ upon substitution as,
\begin{align}
{\bf I} &= \frac{(\nabla u)^t}{\abs{u}^2} \left[ \Delta u  + \frac{\abs{\nabla u}^2}{\abs{u}^2}u - 2 \frac{\nabla u (\nabla u)^t}{\abs{u}^2}u \right] \nonumber \\
&= P^t \left( A^2+ \abs{AP \theta}^2I \right) P \frac{\theta}{r^{2n-1}}  
+  (n-2) \frac{\theta}{r^3}. \label{fullELtwist2}
\end{align}
Now as for a fixed skew-symmetric matrix $B$ by basic differentiation we have 
\begin{equation*}
\nabla \left( \abs{By}^2 \right) = {-2B^2y}, \qquad \nabla \abs{y}^{2n} = 2n\abs{y}^{2n-2}y,
\end{equation*}
it is evident that we can write 
\begin{equation}
- \nabla \left( \frac{\abs{By}^2}{2n\abs{y}^{2n}}  \right) = \frac{B^2y}{n\abs{y}^{2n}} + \frac{\abs{By}^2y}{\abs{y}^{2n+2}}.
\end{equation}
In particular with $B=P^tAP$ being skew-symmetric, $(\ref{fullELtwist2})$ can be written in the form
\begin{align}
{\bf I} &= \frac{(\nabla u)^t}{\abs{u}^2} \left[ \Delta u  + \frac{\abs{\nabla u}^2}{\abs{u}^2}u - 2 \frac{\nabla u (\nabla u)^t}{\abs{u}^2}u \right] \nonumber \\
&= - \nabla \left( \frac{\abs{P^tAP x}^2}{2n\abs{x}^{2n}} \right) + (n-1) \frac{{P^tA^2 P x}}{ n\abs{x}^{2n}}  +  
- (n-2) \nabla \frac{1}{|x|}. \label{fullELtwist3}
\end{align}
Therefore it is plain that $(\ref{fullELtwist3})$ is a gradient field in $\X$ provided that the term on the right and subsequently the middle 
term, that is, the expression 
\begin{equation}
(n-1)\frac{{P^tA^2 P x}}{n \abs{x}^{2n}}
\end{equation} 
is a gradient field in $\X$. By direct calculations ({\it cf.} \cite{ShT}) this is seen to be the case {\it iff} all the eigenvalues of the skew-symmetric matrix $A$ 
are equal. (Note that in odd dimensions this requirement leads to $A=0$.) As a result here $(\ref{fullELtwist3})$ would be a gradient (indeed $\nabla p$) 
and so the Euler-Lagrange system $(\ref{EL})$ is satisfied by the twist $u$.

Now using the representation $A= RSR^t$ for some $S \in (\beta(b)-\beta(a))^{-1} \mathbb{L}$ and writing $S = \lambda J$ where, $J$ is the $n \times n$ 
block diagonal matrix: $J=0$ when $n$ is odd and $J= {\rm diag} \left({\bf A}_1, \cdots, {\bf A}_{n/2} \right)$ when $n$ is even, i.e., 
\begin{align}
J = \left[\begin{array}{cccc}
\mathbf{A}_1 & 0&\cdots & 0\\
0& \mathbf{A}_2 &\cdots & 0\\
\vdots & &\ddots &\vdots \\
0& 0 & \cdots & \mathbf{A}_{n/2}\\
\end{array}\right] 
\qquad 
{\bf A}_j= \begin{bmatrix}
0 & -1 \\
1 & 0 
\end{bmatrix} 
\end{align}
it is required that $\lambda (\beta(b)-\beta(a)) J \in \mathbb{L}$. But invoking the lattice structure of $\mathbb{L}$ this can happen {\it iff}
\begin{equation}
\lambda = \frac{2\pi k}{\beta(b)-\beta(a)}, \qquad k \in \Z,
\end{equation} 
and thus 
\begin{align}
u(x) = R\exp \left( -2k\pi \frac{\beta(r)-\beta(a)}{\beta(b)-\beta(a)} J \right) R^tx.
\end{align}
Noticing that here we have
\begin{align}
\frac{\beta(r)-\beta(a)}{\beta(b)-\beta(a)} &= \frac{\log \left(r/a\right)}{\log \left(b/a\right)},
\end{align}
for $n=2$ and 
\begin{align}
\frac{\beta(r)-\beta(a)}{\beta(b)-\beta(a)} = \frac{\left(r/a\right)^{2-n}-1}{\left(b/a\right)^{2-n} - 1 },
\end{align}
for even $n \ge 4$ respectively we obtain the representation  
\begin{align}
u(x) &= R\exp \left( -g(r) J \right) R^tx = \exp \left( -g(r) A \right) x,
\end{align}
where we have set $A= R J R^t$ and the angle of twist function $g=g(r)$ is given by \eqref{rotation-angle-2d-equation} for $n=2$ and 
\eqref{twistsevendim} for even $n \ge 4$ respectively. 
For odd $n \ge 3$ as shown $A=0$ and so the only twist solution to \eqref{EL} is the trivial solution $u \equiv x$.

\section{Symmetrisation as a means of energy reduction on ${\mathcal A}_\phi(\X)$ when $n=2$} 
\setcounter {equation}{0}

Recall that the space of admissible maps ${\mathcal A}_\phi(\X)$ consists of maps $u \in W^{1,2}(\X, \R^2)$ satisfying the incompressibility condition 
$\det \nabla u = 1$ {\it  a.e.} in $\X$ and $u|_{\partial \X} = \phi$. Also as mentioned earlier due to a Lebesgue-type monotonicity every such map is 
continuous on the closed annulus $\overline \X$ and using degree theory the image of the closed annulus is again the closed annulus itself; hence, the 
"{\it embedding}" 
\begin{equation}
{\mathcal A}_\phi(\X) = \bigcup_{k \in \Z} A_k \subset {\mathfrak A}(\X), 
\end{equation} 
where the components $A_k$ here are as defined by (\ref{AK-L2}). For the sake of future calculations it is useful to write (\ref{degree-2d-equation}) as 
\begin{equation} \label{degkZ}
{\bf deg} (u|u|^{-1}) = \frac{1}{2\pi} \int_a^b \frac{u \times u_r}{\abs{u}^2} \, dr = k \in \Z,
\end{equation} 
where $\abs{x}=r$. (Note that we adopt the convention that in two dimensions the cross product is a scalar and not a vector.) When $u$ is a twist map, 
specifically, $u=Q[g] x$ the integral reduces to $g(b)-g(a) = 2\pi k$ where as before $g=g(r)$ is the angle of rotation function. 
We now proceed by reformulating the ${\mathbb F}$ energy of an admissible map $u \in {\mathcal A}_\phi(\X)$ in a more suggestive way. Indeed 
switching to polar co-ordinates it is seen that 
\begin{equation} \label{2-norm-grad-equation}
\abs{\nabla u}^2  = \abs{u_r}^2 + \frac{1}{r^2} \abs{u_{\theta}}^2 
\end{equation} 
where 
\begin{align*}
|u_r|^2 = \frac{(u \cdot u_r)^2 + (u \times u_r)^2}{|u|^2}, \qquad 
|u_\theta|^2 = \frac{(u \cdot u_{\theta})^2 + (u \times u_{\theta})^2}{|u|^2}.
\end{align*}
Next we note that 
\begin{equation*}
(|u|_r)^2  = \frac{(u \cdot u_r)^2}{|u|^2}, \qquad (|u|_\theta)^2  = \frac{(u \cdot u_\theta)^2}{|u|^2}.
\end{equation*} 
Hence the gradient term on the left in (\ref{2-norm-grad-equation}) can be expressed as  
\begin{align} 
\abs{\nabla u}^2 & = \frac{(u \cdot u_r)^2 + (u \times u_r)^2}{\abs{u}^2} + \frac{(u \cdot u_{\theta})^2 + (u \times u_{\theta})^2 }{r^2\abs{u}^2} \nonumber \\
& = \abs{\nabla \abs{u}}^2 + \frac{(u \times u_r)^2}{\abs{u}^2} + \frac{(u \times u_{\theta})^2 }{r^2\abs{u}^2}.
\end{align} 
From this we therefore obtain the ${\mathbb F}$ energy as
\begin{align}
{\mathbb F} [u; \X] &= \frac{1}{2}\int_\X \frac{\abs{ \nabla u}^2}{\abs{u}^2 }  \, dx 
= \frac{1}{2}\int_0^{2\pi} \int_a^b \frac{\abs{ \nabla u}^2}{\abs{u}^2 } r \, dr d\theta   \nonumber \\
&= \frac{1}{2}\int^{2\pi}_0 \int_a^b \left[ \frac{\abs{\nabla \abs{u}}^2}{\abs{u}^2} +  \frac{(u \times u_r)^2}{\abs{u}^4} + 
\frac{(u \times u_{\theta})^2 }{r^2\abs{u}^4} \right] r \, dr d\theta.
\end{align}
Let us first state the following useful identity that will be employed in obtaining a fragment of the lower bound on the energy: For $u \in {\mathcal A}_\phi(\X)$ and 
$a.e.$ $r\in[a,b]$,
\begin{align}
\int_0^{2\pi} \frac{u(r,\theta)\times u_{\theta}(r,\theta)}{\abs{u}^2} \, d\theta = 2\pi. \label{postponed}
\end{align}
The proof of this identity is postponed until later on in Section \ref{SecSeven} ({\it cf.} Proposition \ref{invariant}). 
Now assuming this for the moment an application of Jensen's inequality gives, again for $a.e.$ $r\in[a,b]$,
\begin{align}
\frac{1}{2\pi}\int_0^{2\pi} \frac{(u\times u_{\theta})^2 }{\abs{u}^4} \, d\theta 
&\geq \left(\frac{1}{2\pi} \int_{0}^{2\pi}\frac{u(r,\theta)\times u_{\theta}(r,\theta)}{\abs{u}^2} \, d\theta \right)^2=1.
\end{align}
Hence it is plain that 
\begin{align}
\int^{2\pi}_0 \int_a^b \frac{(u \times u_{\theta})^2 }{r^2\abs{u}^4} r \, dr d\theta \geq 2\pi \ln(b/a).
\end{align}
Therefore we have the following lower bound on the ${\mathbb F}$ energy of an admissible map $u$:
\begin{align}
{\mathbb F} [u; \X] & \geq \pi \ln(b/a) + \frac{1}{2}\int^{2\pi}_0 \int_a^b \left[ \frac{\abs{\nabla \abs{u}}^2}{\abs{u}^2} 
+ \frac{(u \times u_r)^2}{\abs{u}^4} \right] r \, dr d\theta. \label{lower-bound-u-equation}
\end{align}
Interestingly here we have equality only for twist maps and so outside this class the inequality is strict (for more on questions of uniqueness 
{\it see} \cite{CT}). The next task is to show that by using a basic "{\it symmetrisation}" in ${\mathcal A}_\phi(\X)$ we can reduce the energy 
which will then be the main ingredient in the proof of the result.

\begin{prop}\label{sym} $($Symmetrisation$)$ Let $u \in {\mathcal A}_\phi(\X)$ be an admissible map and associated with $u$ define the angle of rotation 
function $g=g(r)$ by setting 
\begin{equation} \label{sym2}
g(r) = \frac{1}{2\pi}\int_a^r \int_0^{2\pi} \frac{u \times u_r}{\abs{u}^2} \, d\theta dr, \qquad a \le  r \le b. 
\end{equation}
Then the twist map defined by $\bar u (x) = Q[g] x$ with $Q=Q[g]={\bf R}[g]$ has a smaller ${\mathbb F}$ energy than the original map $u$, that is, 
\begin{equation}
{\mathbb F}[\bar u; \X] \le {\mathbb F}[u; \X]. 
\end{equation} 
Furthermore if $u \in A_k$ then the symmetrised twist map $\bar u$ satisfies $\bar u \in A_k$. Thus the homotopy classes $A_k$ are invariant under 
symmetrisation. 
\end{prop}

\begin{proof}
Clearly the symmetrised twist map $\bar u$ is in the same homotopy class as $u$ since by definition $g \in W^{1,2}[a, b]$ satisfies 
$g(a)=0$ and 
\begin{equation}
g(b) = \frac{1}{2\pi}\int_a^b \int_0^{2\pi} \frac{u \times u_r}{\abs{u}^2} \, d\theta dr = 2 \pi k.
\end{equation} 
Therefore $\bar u \in A_k$ as a result of \eqref{degkZ}. Next the $\mathbb{F}$ energy of $\bar{u}$ satisfies the bound
\begin{align}
{\mathbb F} [\bar{u}; \X] - 2\pi \log(b/a) &= \pi \int_a^b |\dot Q(r)|^2 r \, dr \nonumber \\
&= \pi \int_a^b |\dot g(r)|^2 r \, dr \nonumber \\
&= \pi \int_a^b \left[\frac{1}{2\pi}\int_0^{2\pi} \frac{u \times u_r}{\abs{u}^2} d\theta\right]^2 r \, dr  \nonumber \\
&\leq \frac{1}{2}\int_0^{2\pi}\int_a^b \frac{(u \times u_r)^2}{\abs{u}^4} r \, drd\theta
\end{align}
where the last line is a result of Jensen's inequality. Therefore by referring to (\ref{lower-bound-u-equation}) all that 
is left is to justify the inequality
\begin{equation}
 2\pi \log(b/a) \leq \int_\X \frac{\abs{\nabla \abs{u}}^2}{\abs{u}^2} \, dx.
\end{equation}

Towards this end we use the isoperimetric inequality in the context of sets of finite perimeter and the coarea formula in the context 
of Sobolev spaces: For real-valued $f$ and non-negative Borel $g$: 
\begin{equation}
\int_\X g(x) |\nabla f| \, dx = \int_\R \int_{\lbrace f=t \rbrace} g(x) \, d\mathcal{H}^1(x) \, dt.
\end{equation}
({\it See}, e.g., \cite{Brothers, MSZ}.) Then upon taking $f = |u| \in W^{1,2}(\X) \cap {\bf C}(\overline \X)$ and $g = 1/ |u|^2$ this gives 
\begin{align} \label{coareaSob}
\int_\X \frac{\abs{\nabla \abs{u}}}{\abs{u}^2} \, dx &= \int_a^b \left( \int_{\lbrace \abs{u} = t \rbrace} d\mathcal{H}^1 \right) \frac{dt}{t^2} 
= \int_a^b \mathcal{H}^1(\lbrace \abs{u}=t \rbrace) \, \frac{dt}{t^2}. 
\end{align}

Now since the level sets $E_t = \{x \in \X : \abs{u(x)} \leq t \}$ and $F_t=\{x \in \X : \abs{x} \leq t \}$ enclose the same area due to 
$\det \nabla u = 1$ {\it a.e.} (we can consider $u$ as extended by identity inside $\{|x|<a\}$) an application of the isoperimetric 
inequality gives $2\pi t = \mathcal{H}^1(\lbrace \abs{x}=t \rbrace) = {\mathcal H}^1(\partial^\star F_t) 
\leq {\mathcal H}^1(\partial^\star E_t) = \mathcal{H}^1(\lbrace \abs{u}=t \rbrace)$ for $a.e.$ $t\in [a,b]$ ({\it cf.}, e.g., 
\cite{Brothers, Federer}). Thus substituting in \eqref{coareaSob} results in the lower bound
\begin{align}
\int_\X \frac{\abs{\nabla \abs{u}}}{\abs{u}^2} \, dx & = \int_a^b \mathcal{H}^1(\lbrace \abs{u}=t \rbrace) \, \frac{dt}{t^2} \\
& \geq  \int_a^b \mathcal{H}^1(\lbrace \abs{x}=t \rbrace) \, \frac{dt}{t^2} = \int_a^b 2\pi t \, \frac{dt}{t^2} = 2\pi \log (b/a). \nonumber 
\end{align}

Finally we arrive at the conclusion by noting that $u$ and $\phi$ have the same distribution function, that is, again as a result of the pointwise 
constraint $\det \nabla u =1$ {\it a.e.} in $\X$: 
$\alpha_u(t) = |\lbrace x \in \X : \abs{u(x)} \geq t \rbrace | = | \lbrace x \in \X : \abs{x} \geq t \rbrace | = \alpha_\phi(t)$ and therefore
\begin{equation*}
\int_\X \frac{dx}{\abs{u}^2} = \int_a^\infty -2 \alpha_u(t) \, \frac{dt}{t^3} + \frac{\abs{\X}}{a^2} 
= \int_a^{\infty} -2 \alpha_\phi(t) \, \frac{dt}{t^3} + \frac{\abs{\X}}{a^2}
= \int_\X \frac{dx}{\abs{x}^2} = 2\pi \log(b/a).  
\end{equation*}

Now putting all the above together, a final application of H\"older inequality gives,
\begin{align}
 \left(2\pi \log(b/a) \right)^2 &\leq \left( \int_\X \frac{\abs{\nabla \abs{u}}}{\abs{u}^2} \, dx \right)^2 \nonumber \\
&\leq \int_\X \frac{\abs{\nabla \abs{u}}^2}{\abs{u}^2} \, dx \times \int_\X \frac{dx}{\abs{u}^2} \nonumber \\
&= 2 \pi \log(\frac{b}{a})\int_\X \frac{\abs{\nabla \abs{u}}^2}{\abs{u}^2} \, dx
\end{align} 
and thus eventually we have 
\begin{equation*}
2\pi \log(b/a) \leq \int_\X \frac{\abs{\nabla \abs{u}}^2}{\abs{u}^2} \, dx 
\end{equation*} 
and so the conclusion follows. 
\end{proof}

\section{The ${\mathbb F}$ energy and connection with the distortion function}
\setcounter{equation}{0}

In this section we delve into the relationship between the energy functional ${\mathbb F}$ in (\ref{F-energy-equation}) and the notions 
of distortion function and energy of geometric function theory. In particular we show that in two dimensions twist maps have minimum 
distortion among all incompressible Sobolev homeomorphisms of the annulus with identity boundary values in any given homotopy 
class. To fix notation and terminology let $U, V \subset \R^n$ be open sets and 
$f \in W_{loc}^{1,1}(U, V)$. Then $f$ is said to have finite outer distortion {\it iff} 
there exists measurable function $K=K(x)$ with $1\leq K(x)<\infty$ such that 
\begin{equation}
\abs{\nabla f(x)}^n \leq n^{n/2} K(x) \det \nabla f (x).
\end{equation}
The smallest such $K$ is called the {\it outer} distortion of $f$ and denoted by $K_O(x, f)$. Note that here $|A| = \sqrt{{\rm tr} A^tA}$ is 
the Hilbert-Schmidt norm of the $n \times n$ matrix $A$. Naturally $1 \le K_O (x, f) < \infty$ and it measures the deviation of $f$ from 
being conformal. We also speak of the {\it inner} distortion function $K=K_I(x, f)$ defined by the quotient 
\begin{equation}
K_I(x,f) = \frac{n^{-n/2} \abs{{\rm cof}\,\nabla f}^n}{\det ({\rm cof}\, \nabla f)}, \label{innerdistortion}
\end{equation}
when $\det \nabla f (x) \neq 0$ and $K_I(x, f)=1$ otherwise. We define the {\it distortion} energy associated to the inner distortion 
$K_I(x, f)$ \eqref{innerdistortion} by the integral 
\begin{align} \label{distortenergy}
\mathbb{W}[f; U]  = \int_{U} \frac{K_I(x,f)}{\abs{x}^n}\, dx.
\end{align}
Related energies and more have been considered in \cite{IO} with close links to the work in \cite{AIMO}. 
The connection between the ${\mathbb F}$ energy and the distortion energy $\mathbb{W}$ \eqref{distortenergy} is implicit in the following result 
of T.~Iwaniec, G.~Martin, J.~Onninen and K.~Astala \cite{AIMO}. (See also \cite{AIM}, \cite{IO} and \cite{IS}.)

\begin{thm} \label{Iwaniec}
Suppose $f \in W^{1,n}_{loc} (\X, \X)$ is a homeomorphism with finite outer distortion. Assume $K_I$ is $L^1$-integrable over $\X$. 
Then the inverse map $h=f^{-1}: \X \rightarrow \X$ lies in the Sobolev space $W^{1,n}(\X, \X)$. 
Furthermore  
\begin{equation}
n^{-\frac{n}{2}} \int_{\X} \frac{\abs{ \nabla h(y)}^n}{\abs{h(y)}^n} \, dy = \int_{\X} \frac{K(x,f)}{\abs{x}^n} \, dx.
\end{equation}
\end{thm}

\pf The first assertion is Theorem 10.4 pp.~22 of \cite{AIMO}. For the second assertion using definitions we have
\begin{equation*}
n^{n/2} K_I (x, f) = \frac{|{\rm cof} \,\nabla f|^n}{\det ({\rm cof}\,\nabla f)} = |(\nabla f)^{-1}|^n \det \nabla f 
=  |\nabla h(f)|^n \det \nabla f, 
\end{equation*}
and the conclusion follows upon dividing by $|x|^n$ and integrating using the area formula. \hfill $\square$
\\[1 mm]

Now let ${\mathcal A}^n_\phi(\X) =\{u \in W^{1,n}(\X; \R^n) : \det \nabla u =1 \mbox{ {\it a.e.} in $\X$ and $u|_{\partial \X} =\phi$} \}$. 
As in Section \ref{SecTwo} each $u$ in ${\mathcal A}^n_\phi(\X)$ admits a representative in ${\mathfrak A}(\X)$. Now restricting to 
homeomorphisms $u \in \mathcal{A}^n_\phi(\X)$ the above theorem gives $v=u^{-1}\in \mathcal{A}^n_\phi(\X)$ and so by the 
incompressibility constraint
\begin{align}
\int_{\X} \frac{K_I(x,u)}{\abs{x}^n} \, dx 
= n^{-\frac{n}{2}} \int_{\X}  \frac{\abs{{\rm cof}\,\nabla u}^n}{\abs{x}^n} \,dx 
= n^{-\frac{n}{2}} \int_{\X}  \frac{\abs{\nabla v}^n}{\abs{v}^n} \,dx.
\end{align}

In the planar case this allows us to relate the distortion energy of a homeomorphism $u$, say, in $A_{k}$ to the 
$\mathbb{F}$ energy of the inverse map $v=u^{-1}$ in $A_{-k}$ through
\begin{equation}
\mathbb{F}[v;\X] = \frac{1}{2} \int_\X \frac{\abs{\nabla v(y)}^2}{\abs{v(y)}^2} \, dy 
=  \frac{1}{2} \int_\X \frac{\abs{\nabla u(x)}^2}{\abs{x}^2} \, dx = \int_{\X} \frac{K_I(x,u)}{\abs{x}^2} \, dx = \mathbb{W}[u;\X]. \label{energyidentity}
\end{equation}

Therefore by showing that twist maps minimise the ${\mathbb F}$ energy within their homotopy classes we have implicitly shown 
that twist maps minimise the distortion energy within their respective homotopy classes of homeomorphisms 
(as the inverse of a twist map is a twist map in {\it opposite} direction and clearly twist maps are homeomorphisms of annuli 
onto themselves).

\begin{thm} The distortion energy ${\mathbb W}$ has a minimiser $u=u(x; k)$ $($$k\in\Z$$)$ among all 
homeomorphisms within $A_k$. The minimiser is a twist map of the form $u =Q[g] x$ 
where $g(r) = 2\pi k \ln(r/a)/\ln(b/a)$.
In particular the minimum energy is given by,
\begin{align}
\mathbb{W}[u;\X] = \int_\X \frac{K_I(x, u)}{|x|^2} \, dx = 2\pi \ln(b/a) + 4\pi^3 k^2/\ln(b/a). 
\end{align}
\end{thm}

\pf That $u=u_k$ minimises ${\mathbb W}$ amongst homeomorphisms in $A_k$ is a result of $u_{-k}= (u_k)^{-1}$ minimising 
$\mathbb{F}$ over $A_{-k}$ (Proposition \ref{sym}) and \eqref{energyidentity}. Indeed arguing indirectly assume there is a 
homeomorphism $v\in A_k$: $\mathbb{W}[v; \X] < \mathbb{W}[u_k; \X]$. Then by \eqref{energyidentity}, 
$\mathbb{F}[v^{-1}; \X]<\mathbb{F}[u_{-k}; \X]$ and this is a contradiction as 
$v^{-1}$, $u_{-k} \in A_{-k}$ while $\mathbb{F}[u_{-k};\X] = \inf_{A_{-k}} \mathbb{F}$. We are thus left with the 
calculation of the $\mathbb{W}$ energy of $u=u_k$. To this end put $v=u^{-1}$: 
\begin{align}
\mathbb{F}[v; \X] =\frac{1}{2}  \int_\X \frac{|\nabla v|^2}{|v|^2} \, dx &= 2\pi \ln (b/a) + \pi \int_a^b r\dot{g}_{-k}(r)^2 \, dr \nonumber \\
&= 2\pi \ln(b/a) + \frac{4\pi^3 k^2}{\ln(b/a)},
\end{align}
and so a further reference to \eqref{energyidentity} completes the proof. \hfill $\square$
\\[1 mm]

In the higher dimensions, i.e. $n>2$, from Theorem \ref{Iwaniec} we can get an analogous identity to \eqref{energyidentity} for homeomorphisms 
$u$ in $\mathcal{A}^n_\phi$. Indeed, with $v=u^{-1}$ 
\begin{equation} 
n \mathbb{F}[v; \X] =  \int_\X \frac{\abs{\nabla v(y)}^n}{\abs{v(y)}^n} \, dy 
 =  \int_\X \frac{\abs{{\rm cof}\,\nabla u(x)}^n}{\abs{x}^n} \, dx = n^{n/2} \mathbb{W}[u; \X], \label{energyidentity2}
\end{equation}
where the energies ${\mathbb F}=\mathbb{F}_n$ and ${\mathbb W}$ are given by,
\begin{align}
\mathbb{F}[v; \X] = \frac{1}{n} \int_{\X} \frac{\abs{\nabla v(y)}^n}{\abs{v(y)}^n} \, dy, \qquad {\mathbb W}[u; \X]= \int_\X \frac{K_I(x, u)}{|x|^n} \, dx.
\end{align}

Therefore again to find minimisers of $\mathbb{W}$ among homeomorphisms in $u\in\mathcal{A}^n_\phi(\X)$ one can follows the lead of $n=2$ and 
consider the energy $\mathbb{F}$ over $\mathcal{A}^n_\phi(\X)$. It is straightforward to see that we have equality in \eqref{energyidentity2} for twist 
maps $u\in\mathcal{A}^n_\phi(\X)$ with the distortion energy of $u= Q(r)x$ given by
\begin{align}
\mathbb{W}[u;\X] = n^{-n/2} \int_{\X} \left(n \abs{x}^{-2} + \abs{\dot{Q}\theta}^2\right)^{n/2} \, dx.   
\end{align}
Now restricting to the particular case of the twist map being ({\it cf.} \cite{ShT2})
\begin{align}
u(x) = \exp \left( -g(r) J \right) x, \qquad x \in \X,  \label{restrictedtwist}
\end{align} 
with $J$ is as in Section \ref{EuLgSec} and $g\in W^{1,n}([a,b])$ the angle of rotation describing the twist. The corresponding distortion energy is
\begin{align*}
\mathbb{W}[u;\X] &= n^{-n/2} \int_{\X} \left(n \abs{x}^{-2} + \abs{\dot{g}}^2\right)^{n/2} \, dx \\
&= \omega_n n^{-n/2} \int_{a}^b \left(n r^{-2} + \abs{\dot{g}}^2\right)^{n/2} r^{n-1} \, dr .   
\end{align*}

Note that in higher dimensions (i.e., $n\geq 3$) as discussed earlier there are only two homotopy classes in ${\mathcal A}^n_\phi(\X)$. 
A twist map $u=Q(r)x$ lies in the non-trivial homotopy class of ${\mathcal A}^n_\phi(\X)$ {\it iff} the twist loop 
$Q=Q(r) \in \C([a,b], \mathbf{SO}(n))$ based at $I$ lifts to a non-closed path $R=R(r) \in \C([a,b],\mathbf{Spin}(n))$ connecting $\pm 1$ 
in $\mathbf{Spin}(n)$ ({\it see} \cite{TA4} for more.)
\footnote{Note that $\mathbf{Spin}(n)$ is the universal cover of $\mathbf{SO}(n)$ and $\{\pm 1\} \subset {\bf Spin}(n)$ is the fibre 
over $I$ under the covering map.}

Likewise a twist $u$ of the form \eqref{restrictedtwist} lies in the non-trivial homotopy class of ${\mathcal A}^n_\phi(\X)$ {\it iff} the angle 
of rotation function $g$ satisfies $g(b)-g(a)=2\pi k$ for some $k$ odd. When $k$ is even the twist map $u$ lies in the trivial homotopy 
class of ${\mathcal A}_\phi(\X)$. 
\footnote{The identity boundary conditions on $u$ dictates that the angle of rotation function must satisfy $g(b)-g(a)=2\pi k$ for some $k\in\Z$.}

\section{Other variants of the Dirichlet energy} \label{SecSeven}
\setcounter {equation}{0}

The goal of this section is to establish various energy bounds and identities, when $n=2$, by invoking the regularity and the measure preserving 
constraints satisfied by the elements of $\mathcal{A}_{\phi}(\X)$. These inequalities will ultimately lead to useful results for extremisers 
and minimisers of variants of the Dirichlet energy in homotopy classes of ${\mathcal A}_\phi(\X)$. We begin with the following identity.
\begin{prop}\label{prop1}
For $\Phi \in {\bf C}^1[a,b]$ the integral identity 
\begin{equation}
\int_\X \Phi(|u|) \, dx = \int_\X \Phi(|x|) \, dx
\end{equation}
holds for all $u \in {\mathcal A}_\phi(\X)$.
\end{prop}
\begin{proof} Denoting by $\alpha_u=\alpha_u(t)$ the distribution function of $|u|$ we can write using basic considerations and invoking the standard properties of distribution functions
\begin{align}
\int_\X \Phi(\abs{u}) \, dx & = \int_{\X} \int_a^{\abs{u}} \dot \Phi (t) \, dt dx + \Phi(a) \abs{\X} 
= \int_a^b \dot \Phi(t) \alpha_u(t) \, dt + \Phi(a) \abs{\X} \nonumber \\
& = \int_a^b \dot \Phi(t) \alpha_x(t) \, dt + \Phi(a) \abs{\X} 
= \int_\X \int_a^{\abs{x}} \dot \Phi (t) \, dt dx + \Phi(a) \abs{\X} \nonumber \\
&= \int_\X \Phi(\abs{x}) \, dx,
\end{align}
which is the required conclusion. 
\end{proof}

We now collect a few more results which will be needed for the proof of our main theorem at the end of this section.

\begin{prop}\label{ann} 
Let $\Phi \in {\bf C}^1[a,b]$ and pick $u\in \mathcal{A}_\phi(\X)$. Consider the continuous closed curve $\gamma(\theta) = u(r,\theta)$ where 
$a<r<b$ is fixed. Then the integral identity
\begin{align}
\int_0^{2\pi} \Phi(\abs{u})^2(u\times u_{\theta}) \, d\theta \, \bigg|^r_a = 2\int_{\X_r} 
\left[\abs{u} \Phi(\abs{u})\dot{\Phi}(\abs{u})+\Phi(\abs{u})^2 \right] \,dx, \label{formula1}
\end{align}
holds for almost every $r\in (a,b)$, where $\X_r=\X[a,r]=\lbrace x\in\R^2:a<\abs{x}<r \rbrace$.
\end{prop}

\pf We shall justify the assertion first when $u$ is a sufficiently smooth diffeomorphism and then pass on to the general case by invoking a suitable 
approximation argument. Towards this end consider first the case where $u$ is a smooth diffeomorphism with $u \equiv x$ on $\partial \X$. [Here 
$u$ need not satisfy the incompressibility condition in ${\mathcal A}_\phi(\X)$.] Let $\alpha$ denote 
the $1$-form 
\begin{equation}
\alpha = \Phi(\abs{x})^2 \left(x_1dx_2-x_2dx_1\right).
\end{equation} 
Then by a rudimentary calculation the pull-back of $\alpha$ under the ${\bf C}^\infty$ curve $\gamma$ is given by 
\begin{align}
\gamma^*\alpha = \Phi(\abs{u})^2 (u\times u_{\theta} ) \, d\theta.
\end{align}
Hence contour integration and basic considerations lead to the integral identity  
\begin{align}
\int_{\gamma} \alpha = \int_0^{2\pi} \Phi(\abs{u})^2 (u\times u_{\theta}) \, d\theta \, \bigg|_r.
\end{align}

Note that the ${\bf C}^\infty$ curve $\gamma$ here is diffeomorphic to $\Sp^1$ and as such by the Jordan-Sch\"oenflies theorem $\gamma$ is the boundary 
of some bounded region $C_\gamma\subset \R^2$ diffeomorphic to the unit ball $\B_1$. In particular due to the boundary conditions on $u$ we have that 
$a < \abs{\gamma}(\theta)=\abs{u}(r,\theta)$ when $a<r$ and so as a result $C_a^\gamma = \lbrace x \in\R^2: a<\abs{x}<\abs{\gamma} \rbrace\subset\X$ 
with boundaries of $\partial{\B}_a$ and $\gamma$. Hence application of Stoke's theorem gives
\begin{align}
\int_{C_a^\gamma} d\alpha = \int_{\partial C_a^\gamma} \alpha 
= \int_{\gamma} \alpha - \int_{|x|=a} \alpha 
= \int_0^{2\pi} \Phi(\abs{u})^2 u\times u_{\theta} \, d\theta \,\bigg |_a^r, \label{formula12}
\end{align}
for all $r\in[a,b]$. Next again by a rudimentary calculation we obtain that the pull-back of $d\alpha$ is,
\begin{align}
d\alpha = 2 \det \nabla u \left[\abs{u} \Phi(\abs{u})\dot{\Phi}(\abs{u})+\Phi(\abs{u})^2 \right]  dx_1 \wedge d x_2. \label{formula13}
\end{align}
Hence from \eqref{formula12} and \eqref{formula13} it follows that for all $r\in[a,b]$,
\begin{align}
2\int_{\X_r}\left[\abs{u} \Phi(\abs{u})\dot{\Phi}(\abs{u})+\Phi(\abs{u})^2 \right]  \det \nabla u \, dx 
&= \int_0^{2\pi} \Phi(\abs{u})^2 u\times u_{\theta} \, d\theta  \, \bigg|_a^r. \label{formula16}
\end{align}
Now pick an arbitrary $u$ in $\mathcal{A}_\phi(\X)$. By approximation, e.g., using Theorem 1.1 in \cite{HM-C} there is a sequence 
of ${\bf C}^{\infty}$ diffeomorphisms $(v^k)$ so that $v^k-u \in W^{1,2}_0(\X,\R^2)$ with $v_k \rightarrow u$ uniformly on $\overline \X$ 
and strongly in $W^{1,2}$. Hence,
\begin{align}
f_k:= \Phi(\abs{v^k})^2 \frac{v^k \times v^k_\theta}{|x|} \rightarrow \Phi(\abs{u})^2 \frac{u \times u_\theta}{|x|} =:f, 
\end{align}
{\it a.e.} in $\X$. Note that $\abs{f_k}\leq c \abs{v^k_\theta}$, $\abs{f}\leq c \abs{u_\theta}$ for some $c>0$ and so $f_k,f \in L^2(\X)$ 
and by virtue of $v_k \rightarrow u$ in $W^{1,2}$ and dominated convergence, for each $r\in(a,b)$ and $0<\delta<b-r$, we have
\begin{align}
\int_r^{r+\delta}\int_0^{2\pi} \Phi(\abs{v^k})^2(v^k\times v^k_{\theta}) \, d\theta dr \rightarrow 
\int_r^{r+\delta}\int_0^{2\pi} \Phi(\abs{u})^2(u\times u_{\theta}) \, d\theta dr. \label{formula14}
\end{align}
In a similar spirit we have (suppressing the arguments of $\Phi$ for brevity) 
\begin{align}
h_k:=\left[\abs{v^k} \Phi\dot{\Phi}+\Phi^2 \right] \det \nabla v^k  \rightarrow \left[\abs{u} \Phi\dot{\Phi}+\Phi^2 \right]  \det \nabla u =:h,
\end{align}
{\it a.e.} in $\X$. Again since $\abs{h_k} \leq c \abs{v^k_{x_1}} \abs{v^k_{x_2}}$, $\abs{h}\leq c \abs{u_{x_1}} \abs{u_{x_2}}$ for some 
$c>0$ we have $h_k,h\in L^1(\X)$ and so by dominated convergence
\begin{align}
\int_{\X_r}\left[\abs{v^k} \Phi\dot{\Phi}+\Phi^2 \right] \det \nabla v^k \, dx \rightarrow 
\int_{\X_r}\left[\abs{u} \Phi\dot{\Phi}+\Phi^2 \right]  \det \nabla u \, dx. \label{formula15}
\end{align}
Now combining \eqref{formula16} and \eqref{formula14} together with the fact that 
$u=v^k=\phi$ on $\partial \X$ it follows that
 \begin{align*}
2\int_r^{r+\delta} \int_{\X_r}\left[\abs{v^k} \Phi\dot{\Phi}+\Phi^2 \right] \det \nabla v^k \, dx dr \rightarrow 
 \int_r^{r+\delta}\int_0^{2\pi} \Phi(\abs{u})^2(u\times u_{\theta}) \, d\theta \, \bigg |_a^r dr.
 \end{align*}
Moreover \eqref{formula15} and a final application of dominated convergence gives
\begin{align*}
2 \int_r^{r+\delta} \int_{\X_r}\left[\abs{u} \Phi\dot{\Phi}+\Phi^2 \right]  \det \nabla u \, dx dr 
= \int_r^{r+\delta} \int_0^{2\pi}  \Phi(\abs{u})^2(u\times u_{\theta}) \, d\theta \, \bigg |_a^r dr.
\end{align*}
Therefore the result follows by recalling that $\det \nabla u = 1 $ $a.e.$ in $\X$ and applying the Lebsegue differentiation 
theorem, i.e., dividing by $\delta$ and letting $\delta\searrow 0$. \hfill $\square$
\\[1 mm]

Note that we can write the conclusion of the above proposition, namely, the integral identity \eqref{formula1} in a shorter and somewhat more suggestive form 
\begin{align}
\int_0^{2\pi} \Phi(\abs{u})^2(u\times u_{\theta}) \, d\theta \, \bigg|_a^r = \int_{\X_r} \abs{u}^{-1}\dot{\Gamma}(\abs{u}) \,dx,  \label{formula11}
\end{align} 	
where $\Gamma (t) = t^2 \Phi(t)^2$ for $\Phi \in {\bf C}^1[a,b]$. With this result and formulation at our disposal we can now prove the earlier 
relation $(\ref{postponed})$ as a specific proposition. 
\begin{prop}\label{invariant}
{Taking $\Phi(t)= 1/t$ in the above gives the integral identity }
\begin{align}
\int_0^{2\pi}  \frac{{u(r,\theta)\times u_{\theta}(r,\theta)}}{\abs{u}^2} \, d\theta = 2\pi, \qquad a.e. \,\, r\in[a,b]. 
\end{align}
\end{prop}
\begin{proof}
When $\Phi(t) = 1/t$ it can be easily seen that $\dot{\Gamma}(t)=0$ and therefore Proposition \ref{ann} gives that,
\begin{align}
\int_0^{2\pi} \frac{(u\times u_{\theta})}{\abs{u}^2} \, d\theta \, \bigg|_a^r = 0, \quad a.e. \,\,r\in [a,b]. \label{formula18}
\end{align}
Then recalling that $u(x)=x$ on $\partial \X$, i.e. when $\abs{x}=a$, it can be seen that the integral over the inner boundary is $2\pi$ and 
hence from \eqref{formula18} this implies that,
\begin{align}
\int_0^{2\pi} \frac{(u\times u_{\theta})}{\abs{u}^2} \, d\theta \,  = 2\pi, \quad a.e. \,\, r\in [a,b], 
\end{align}
which completes the proof.
\end{proof}

\begin{prop}\label{invariant4} 
Suppose $\Gamma \in {\bf C}^2[a,b]$ such that $\dot{\Gamma}(t)/t$ is a monotone increasing function. Then for almost every $r\in[a,b]$ we have that
\begin{align}
\int_0^{2\pi} \Gamma(\abs{u})^2 \frac{({u(r,\theta)\times u_{\theta}(r,\theta)})^2 }{\abs{u}^4} \, d\theta \geq 2\pi \Gamma(r)^2. 
\end{align}
\end{prop}

\begin{proof} First we note that by Proposition \ref{ann} we have for {\it a.e.} $r \in [a, b]$ that 
\begin{align}
\int_0^{2\pi} \Gamma (\abs{u}) \frac{u\times u_{\theta}}{\abs{u}^2} \, d\theta \, \bigg |_a^r= \int_{\X_r} \abs{u}^{-1} \dot{\Gamma}(\abs{u})  dx. 
\end{align}
As $\dot{\Gamma}(t)/t$ is monotone increasing and $\abs{u}$ and $\abs{x}$ share the same distribution function $\alpha_{x}(t)$ we have that, 
\begin{align}
\int_{\X^b_r} \abs{u}^{-1} \dot{\Gamma}(\abs{u}) \, dx & = \int_a^b \frac{d}{d t} \left(\frac{\dot{\Gamma}(t)}{t}\right) 
\int_{\X} \chi_{\lbrace x \in\X_r^b : \abs{u}(x)>t \rbrace} \,dx \, dt + \abs{\X_r^b} \frac{\dot{\Gamma}(a)}{a}  \nonumber\\
&  \leq\int_a^b \alpha_{x}(t) \frac{d}{d t} \left(\frac{\dot{\Gamma}(t)}{t}\right) dt + \abs{\X_r^b} \frac{\dot{\Gamma}(a)}{a} \nonumber \\
& = \int_{\X_r^b}    \abs{x}^{-1} \dot{\Gamma}(\abs{x}) \, dx \nonumber \\
&= 2\pi \left[ \Gamma(b)-\Gamma(r)\right].\label{Gamma1}
\end{align}
In the above $\X_r^b =\lbrace x \in\X : r< \abs{x}<b \rbrace$. Now since,
\begin{align}
2\pi \left[ \Gamma(b)-\Gamma(a)\right]= \int_{\X}   \abs{x}^{-1} \dot{\Gamma}(\abs{x}) dx = \int_{\X} \abs{u}^{-1} \dot{\Gamma}(\abs{u})  dx, \label{Gamma2}
\end{align}
by Proposition \ref{prop1} we obtain upon using \eqref{Gamma1} in \eqref{Gamma2} that,
\begin{align}
 \int_{\X_r} \abs{u}^{-1} \dot{\Gamma}(\abs{u})  dx \geq \int_{\X_r} \abs{x}^{-1} \dot{\Gamma}(\abs{x})\, dx = 2\pi \left[\Gamma(r)-\Gamma(a)\right]. 
\end{align}
Therefore,
\begin{align}
\int_0^{2\pi} \Gamma (\abs{u}) \frac{u\times u_{\theta}}{\abs{u}^2} \, d\theta \, \bigg |_a^r\geq 2\pi  \left[\Gamma(r)-\Gamma(a)\right], 
\end{align}
but from the identity boundary conditions on $u$ we know that,
\begin{align}
\int_0^{2\pi} \Gamma (\abs{u}) \frac{u\times u_{\theta}}{\abs{u}^2} \, d\theta \, \bigg |_a = 2\pi\Gamma(a), 
\end{align}
and therefore,
\begin{align}
\int_0^{2\pi} \Gamma (\abs{u}) \frac{u\times u_{\theta}}{\abs{u}^2} \, d\theta\, \bigg |_r \geq 2 \pi \Gamma(r), \qquad a.e. \,\, r\in[a,b].
\end{align}
The result then follows from an application of Jensen's inequality. 
\end{proof}

With the aid of these bounds we can now move on to the main goal of the section, namely, formulating and proving minimality for twist maps 
in homotopy classes of ${\mathcal A}_\phi(\X)$ for a larger class of energies than those considered earlier.

\begin{thm}\label{mainthm}
Let $\X=\X[a, b] \subset \R^2$ and let $\mathbb{H}={\mathbb H}[u; \X]$ denote the energy functional,
\begin{align}
\mathbb{H}[u;\X] &= \frac{1}{2} \int_{\X} \Phi (\abs{u}) \left[ \abs{\nabla \abs{u}}^2  + 
\frac{(u\times u_{\theta})^2}{r^2\abs{u}^2} \right] +\frac{(u\times u_r)^2}{\abs{u}^4} \, dx,
\end{align} 
where $u$ lies in $\mathcal{A}_{\phi}(\X)$ and $\Phi (t) = t^{-2}\Gamma(t)^2$ with $\Gamma(t)\in {\bf C}^2[a,b]$ such that $\dot{\Gamma}(t)/t$ 
is monotone increasing. Then for any $u \in A_k$ $($$k \in \Z$$)$ there exists a twist map 
$\bar{u} = \bar u_k=Q[g] x$ defined by the same symmetrisation as in $(\ref{sym})$ such that,
\begin{align}
\mathbb{H}[\bar{u}; \X] \leq \mathbb{H}[u; \X]
\end{align}
whilst $\bar u \in A_k$. 
\footnote{Note that taking $\Phi(t)= 1/t^2$, i.e., $\Gamma(t) = 1$, gives $\mathbb{H}=\mathbb{F}$.}
\end{thm}

\begin{proof} As the first step in the proof we wish to prove the inequality 
\begin{align} \label{desired-main-thm-ineq}
\int_{\X} \Phi(\abs{x}) \, dx  = \int_{\X} \Phi(\abs{\bar{u}(x)}) \abs{\nabla\abs{\bar{u}(x)}}^2 \, dx \leq \int_{\X} \Phi(\abs{u}) \abs{\nabla\abs{{u}(x)}}^2 \, dx. 
\end{align}

In order to do this we again need to apply the isoperimetric inequality, the coarea formula for Sobolev functions as in the proof of Proposition \ref{sym} 
and then the integral identity $(\ref{prop1})$. Thus we proceed by writing 
\begin{align}
\int_{\X} \Phi(\abs{u}) \abs{\nabla\abs{{u}(x)}} \, dx &=  \int_a^b \Phi(t)\mathcal{H}^1(\abs{u}=t) \, dt \nonumber \\
& \geq \int_a^b \Phi(t)\mathcal{H}^1(\abs{x}=t) \, dt \nonumber \\
&= \int_{\X} \Phi(\abs{x}) \, dx. 
\end{align} 
Therefore it follows from basic considerations that 
\begin{align}
\left(\int_{\X} \Phi(\abs{x}) dx\right)^2 &\leq \left(\int_{\X} \Phi(\abs{u}) \abs{\nabla\abs{{u}(x)}} \, dx \right)^2 \nonumber \\
 & \leq \int_{\X} \Phi(\abs{u}) \abs{\nabla\abs{{u}(x)}}^2 dx \int_{\X} \Phi(\abs{u}) \, dx \nonumber \\
 & = \int_{\X} \Phi(\abs{u}) \abs{\nabla\abs{{u}(x)}}^2 dx \int_{\X} \Phi(\abs{x}) \, dx, 
\end{align}
and so rearranging gives the desired inequality (\ref{desired-main-thm-ineq}), namely, 
\begin{align}
\int_{\X} \Phi(\abs{x}) dx \leq \int_{\X} \Phi(\abs{u}) \abs{\nabla\abs{{u}(x)}}^2 \, dx. 
\end{align}
Next we proceed by writing  
\begin{align}
\int_{\X} \Phi(\abs{u})\frac{(u\times u_{\theta})^2}{r^2\abs{u}^2} \, dx &= 
\int_a^b \frac{1}{r} \int_0^{2\pi} \Phi(\abs{u})\frac{(u\times u_{\theta})^2}{\abs{u}^2} \, d\theta dr.
\end{align} 
As ${\Phi(t)} = t^{-2}\Gamma(t)^2$ it follows upon noting Proposition $\ref{invariant4}$ that we have 
\begin{align*}
\frac{1}{2\pi}\int_0^{2\pi}  \Phi(\abs{u})\frac{(u\times u_{\theta})^2}{\abs{u}^2} d\theta &= \frac{1}{2\pi}\int_0^{2\pi} \Gamma(\abs{u})^2\frac{{(u\times u_{\theta})^2}}{\abs{u}^4} d\theta\\
&\geq  \Gamma(r)^2 = \Phi(r) r^2, \quad a.e. \,\, r\in[a,b].
\end{align*}
Hence by combining the above it follows that 
\begin{align}
\int_{\X} \Phi(\abs{u})\frac{(u\times u_{\theta})^2}{r^2\abs{u}^2} \, dx  \geq 2\pi \int_a^b {\Phi(r)}{r} \, dr 
= \int_{\X} \Phi(\abs{\bar{u}})\frac{(\bar{u}\times \bar{u}_{\theta})^2}{r^2\abs{\bar{u}}^2} \, dx.
\end{align}
Therefore with the above at our disposal all that remains is to use the inequality 
\begin{align}
\int_{\X} \frac{(u\times u_r)^2}{\abs{u}^4} \, dx \geq \int_{\X}  \frac{(\bar{u}\times \bar{u}_r)^2}{\abs{\bar{u}}^4} \, dx, 
\end{align}
whose proof proceeds similar to that of Proposition \ref{sym} by using the 
same angle of rotation function (\ref{sym2}) in defining $\bar{u}$. This therefore completes the proof. 
\end{proof}

\section{Measure preserving self-maps and twists on solid tori}
\setcounter {equation}{0}

In this section we propose and study extensions of twist maps to a larger class of domains. Recalling that an $n$-dimensions annulus takes the 
form $\X = [a,b] \times \Sp^{n-1}$ the natural extension here would be domains of the product type $\X=\B^m \times \Sp^{n-1}$ (with $m \ge 1$, 
$n \ge 2$) embedded in $\R^{m+n}$. Twist maps in turn will be suitable measure preserving self-maps of $\overline \X$ that agree with the 
identity map $\phi$ on the boundary $\partial \X$ ({\it see} \cite{TA2, TA4}). To keep the discussion tractable we confine ourselves here to the 
case $m+n=3$. We proceed by first considering the solid torus ${\bf T} \cong \mathbb{B}^2 \times \Sp^1$ embedded in $\R^3$ as (see Fig. 1):
\begin{equation}
{\bf T} = \bigg\{ x=(x_1, x_2, x_3) :   ( \sqrt{x_1^2+x_2^2} - \rho )^2 + x_3^2 = r^2 , \quad 0 \le r < 1 \bigg \rbrace. 
\end{equation}
Here ${\bf T}={\bf T}_\rho$ and the fixed parameter $\rho$ is chosen $\rho>1$ to avoid self-intersection. Now let us set 
$\mu = \sqrt{x_1^2+x_2^2}- \rho$. Then ${\bf T}$ above can be represented as
\begin{align}
0 \le \mu^2 + x_3^2  = r^2 < 1.
\end{align}   
From now on $(\mu, x_3)$ is the preferred choice of co-ordinates for $\B=\mathbb{B}_1^2$  where 
$\B=\lbrace (\mu,x_3)\in \R^2 : \mu^2+x_3^2 < 1 \rbrace$ is unit ball in the $(\mu, x_3)$ plane. In polar co-ordinates we have $\mu = r \cos\theta$, 
$x_3 = r\sin \theta$ and upon noting $\mu = \sqrt{x_1^2+x_2^2}- \rho$ we have $(x_1, x_2)$ as the co-ordinates of a sphere of radius 
$\rho + r\cos\theta$:
\begin{align}
x_1 = (\rho + r\cos\theta) \cos\phi, \quad x_2 = (\rho + r\cos\theta) \sin\phi. 
\end{align}
For our purposes in this section we shall write the above co-ordinate system in the following way,
\begin{align}
\begin{cases}
x_1 = (\mu+\rho)\cos\phi \\
x_2 = (\mu+\rho) \sin\phi\\
x_3 = x_3.
\end{cases}
\end{align}
Now with the above notation in place we can define the desired twist maps on the solid torus $\mathbf{T}$ as
\footnote{Note that the aim here is to seek non-trivial extremising twist maps for the energy functional ${\mathbb F}$ over the admissible 
class of maps ${\mathcal A}_\phi(\mathbf{T})$.},
\begin{align} \label{twist-solid-torus-equation}
u(x) & = Q(\mu,x_3) x, \qquad x \in {\bf T}, 
\end{align}
where the rotation matrix $Q$ in ${\bf SO}(3)$ takes the explicit form 
\begin{align} \label{Q-matrix-equation}
Q(\mu,x_3) = \begin{bmatrix}
\cos g(\mu,x_3) & -\sin g(\mu,x_3) & 0 \\
\sin g(\mu,x_3) & \cos g(\mu,x_3)  & 0 \\
0 & 0 & 1
\end{bmatrix}.
\end{align}

Here the function $g=g(\mu,x_3)$ defines the angle of rotation as in the case for the annulus, however, in this case $g$ depends on the two variables 
$(\mu, x_3)$ and not just one $r=|x|$ as is the case for the annulus. There are two main reasons for this choice of representation of a twist map for 
$\mathbf{T}$, which we describe below. 
\begin{itemize}
\item Firstly in order to be consistent with twist maps for the annulus we require that the rotation matrix is an isometry of the boundary
\begin{align}
\partial{\bf T} = \bigg\{ x=(x_1, x_2, x_3) :   ( \sqrt{x_1^2+x_2^2} - \rho )^2 + x_3^2 = 1  \bigg \rbrace, 
\end{align}
with respect to the metric induced from its embedding in $\R^3$. (This is similar to what was done earlier in the case of an annulus). Then with this in mind the isometries 
of $\partial\mathbf{T}$ are $\mathbf{SO}(3)$ matrices of the form,
\begin{align}
Q = \begin{bmatrix}
\cos \varphi & -\sin \varphi & 0 \\
\sin \varphi & \cos \varphi  & 0 \\
0 & 0 & 1
\end{bmatrix}.
\end{align}
\item Secondly we allow the angle of rotation function $g$ here to depend on the two variables $(\mu, x_3)$ instead of one to incorporate all the "{\it ball}" 
variables in the product structure on ${\bf T}$. Note that $(\mu, x_3)$ collapses into $r=|x|$ in the annulus case as here one deals with an interval 
(a one dimensional ball). 
\end{itemize}
Now in preparation for the upcoming calculations let us denote $y=[x_1,x_2,0]^t$, $\vartheta = y/|y|$ and $g_{\mu}= \partial g / \partial \mu$. 
Then it is easily seen that 
\begin{align}
\begin{cases}
\nabla u = Q + \dot{Q}x \otimes \nabla g,\\
(\nabla u)(\nabla u)^t u = Qx + \langle \nabla g, x \rangle \dot{Q}x,\\
\abs{\nabla u}^2  = 3 +  (\mu+\rho)^2\abs{\nabla g}^2, \\ 
\abs{u}^2 =  (\mu+\rho)^2+x_3^2, \\
\det(\nabla u)  = \det \left(  Q + \dot{Q}x \otimes \nabla g \right)  \\
 \qquad \quad \,\,\,\,\, = 1 + \langle Q^t\dot{Q}x, \nabla g \rangle = 1.
\end{cases}\label{twistcalculations}
\end{align}
Note that the last equality results from the fact that the product $Q^t\dot{Q}$ is skew-symmetric and $\nabla g = [(\mu+\rho)^{-1}x_1 g_\mu,(\mu+\rho)^{-1}x_2 g_\mu, g_{x_3}]^t$. 
Therefore using the above it is seen that the energy for a twist maps is given by
\begin{align}
\mathbb{F}[u; \mathbf{T}] = \frac{1}{2}\int_{\bf T} \frac{|\nabla u|^2}{|u|^2} \, dx 
= \pi\int_{\B} \frac{3+(\mu+\rho)^2\abs{\nabla g}^2}{(\mu+\rho)^2+x_3^2} (\mu+\rho) \, d\mu dx_3. \label{Energytwisttorus}
\end{align}
Here we are using the fact that a change in the co-ordinates $(r,\theta,\phi)\rightarrow (\mu,x_3,\phi)$ results in a Jacobian factor of $\mu+\rho$ 
in the integral, i.e., 
\begin{align*}
\int_0^{2\pi}\int_0^{2\pi}\int_{-1}^1 r(\rho + r\cos \theta ) \, drd\theta d\phi = \int_0^{2\pi}\int_{\B} (\mu+\rho) \, d\mu d{x_3}d\phi.
\end{align*}
Hence $(\ref{Energytwisttorus})$ becomes,
\begin{align}
\mathbb{F}[u;\mathbf{T}] & = \pi \int_{\B}  \frac{(\mu+\rho)^3(g_\mu^2+g_{x_3}^2)}{(\mu+\rho)^2+x_3^2} \, d\mu d x_3 + \frac{3}{2}\int_{\mathbf{T}}\abs{x}^{-2}dx,
\end{align} 
where here the additional absolute constant does not affect the variational structure of ${\mathbb F}$. Now to derive the Euler-Lagrange equation 
associated to the energy integral on the right it suffices to take variations of $g=g(\mu, x_3)$ by some $\varphi \in \mathbf{C}^\infty_c(\B)$. This calculation leads to the following 
divergence form equation:
\begin{align}
\frac{\partial }{\partial \mu} \left( \frac{(\mu+\rho)^3 g_{\mu}}{(\mu+\rho)^2 + x_3^2}  \right)  +  
\frac{\partial }{\partial x_3} \left( \frac{(\mu+\rho)^3 g_{x_3}}{(\mu+\rho)^2 + x_3^2}  \right) = 0, \nonumber \\
\implies {\rm div} \frac{(\mu+\rho)^3 \nabla g}{(\mu+\rho)^2 + x_3^2} =0. \label{Eulertorustwist}
\end{align}

Evidently the identity boundary condition on $u$ translates into $g(z) = 2 k \pi$ for some fixed $k \in \Z$ and all $z=(\mu,x_3)\in \partial \B$. 
Now suppose $g$ solves $(\ref{Eulertorustwist})$. Then by an application of the divergence theorem 
it is seen that the only solution to this boundary value problem is the trivial one, namely, $g(\mu,x_3)=2\pi k$ for all $(\mu,x_3) \in \B$. 
Indeed
\begin{align}
0 &  = \int_{\B} {\rm div} \left(  \frac{(\mu+\rho)^3 \nabla g}{(\mu+\rho)^2 + x_3^2} \right) \, d\mu dx_3 \nonumber \\
& = \int_0^{2\pi}  \frac{(\cos \theta+\rho)^3}{(\cos \theta+\rho)^2 + x_3^2} \frac{\partial g}{\partial r}(1,\theta) \, d\theta. \label{identity1}
\end{align} 
Now again as $g$ solves $(\ref{Eulertorustwist})$ an application of the divergence theorem also gives 
\begin{align*}
 \int_{\B} \frac{(\mu+\rho)^3(g_\mu^2+g_{x_3}^2)}{(\mu+\rho)^2+x_3^2} \, d\mu dx_3 &= 
 \int_{\B} \left[ \frac{(\mu+\rho)^3(g_\mu^2+g_{x_3}^2)}{(\mu+\rho)^2+x_3^2} + g \, {\rm div} \frac{(\mu+\rho)^3\nabla g}{(\mu+\rho)^2 + x_3^2} \right] \, d\mu dx_3 \\
&=  \int_0^{2\pi}  g(1,\theta)\frac{(\cos \theta+\rho)^3}{(\cos \theta+\rho)^2 + x_3^2} \frac{\partial g}{\partial r}(1,\theta) \, d\theta \\
&= 2\pi k  \int_0^{2\pi}  \frac{(\cos \theta+\rho)^3}{(\cos \theta+\rho)^2 + x_3^2} \frac{\partial g}{\partial r}(1,\theta) \, d\theta = 0.
\end{align*}
Note that in obtaining the last identity we have used $(\ref{identity1})$ combined with the boundary condition satisfied by $g$, namely, 
$g(1,\theta)=2\pi k$ for $0\leq \theta \leq 2\pi$. Hence
\begin{align}
\int_\B \frac{(\mu+\rho)^3(g_\mu^2+g_{x_3}^2)}{(\mu+\rho)^2+x_3^2} \, d\mu dx_3 = 0. \label{identity2}
\end{align}

Now since by assumption $\rho>1$ we have $(\rho + \mu) >0$ as $\abs{\mu}=\abs{r\cos\theta} \leq r < 1$ and so $\mu>-1$. 
Hence $(\ref{identity2})$ gives $\abs{\nabla g}^2=0$ and thus $g(\mu,x_3)=2\pi k$; again by invoking the boundary condition on 
$g$. It therefore follows that here we have no non-trivial solutions. Interestingly note that this conclusion stems from one crucial 
difference between the annulus $\X$ and the solid torus $\mathbf{T}$ in that $\X$ has two boundary components whilst 
$\mathbf{T}$ only has one. It was precisely this difference that turned crucial in the application of the divergence theorem.

\begin{thm} There are no non-trivial twist solutions $(\ref{twist-solid-torus-equation})$ to the Euler-Lagrange equations 
associated with the energy functional ${\mathbb F}$ on a solid torus ${\bf T}$.
\end{thm}

\begin{figure}
\begin{center}
\includegraphics{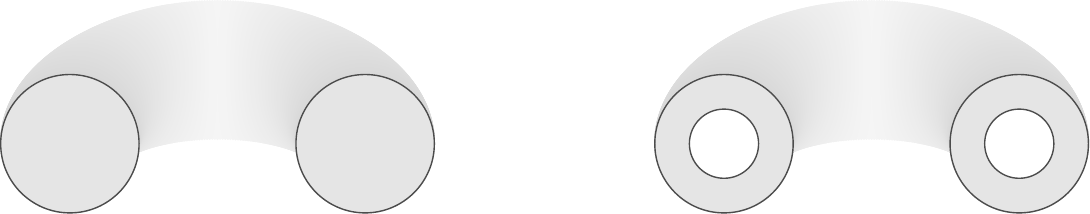}
\caption{The solid torus (left) has a connected space of self-maps $\mathfrak{A}(\mathbf{T})$, whereas for a "{\it thickened}" torus (right) 
the space of self-maps $\mathfrak{A}(\mathbb{T})$ has infinitely many, indeed, $\Z \oplus \Z$ components. (See \cite{TA2} and \cite{TA4} for more.)}
\end{center}
\end{figure}
\section{Twist maps on tori with disconnected double component boundary}

In contrast to what was seen above let us next move on to considering a "{\it thickened}" torus, that is, the domain obtained 
topologically by taking the product of a two-dimensional torus and an interval. Note that here the boundary of the resulting 
domain consists of two disjoint copies of the initial torus and is in particular not connected. Now for definiteness and to fix 
notation let us set ${\mathbb T}={\mathbb T}_\rho$ to be (see Fig. 1)
\begin{align}
{\mathbb T} = \bigg\{ x=(x_1, x_2, x_3)  : ( \sqrt{x_1^2+x_2^2} - \rho)^2 + x_3^2 = r^2 , \quad a < r < 1 \bigg\}.
\end{align}
Here $0<a<1<\rho$ are fixed and the aim is to seek non-trivial extremising twist maps for the energy functional ${\mathbb F}$ 
over the admissible class of maps ${\mathcal A}_\phi({\mathbb T})$. Using the same co-ordinate system as in the earlier case we see 
that the $(\mu,x_3)$ are the co-ordinates of the two dimensional annulus centred at the origin. Additionally for similar reasons to that 
discussed earlier we define twist maps on ${\mathbb T}$ as $u(x) = Q(\mu,x_3) x$ where the rotation matrix $Q=Q(\mu, x_3)$ 
in ${\bf SO}(3)$ is as (\ref{Q-matrix-equation}).

A straightforward calculation shows that the 
energy of a twist maps is given by the integral 
\begin{align}
\mathbb{F}[u; {\mathbb T}] &= \frac{1}{2}\int_{\mathbb T} \frac{|\nabla u|^2}{|u|^2} \, dx 
= \pi\int_{\B_1\backslash \overline{\B}_a} \frac{3+(\mu+\rho)^2\abs{\nabla g}^2}{(\mu+\rho)^2+x_3^2} (\mu+\rho) \, d\mu d x_3 \nonumber \\
& = \pi \int_{\B_1\backslash \overline{\B}_a}\frac{(\mu+\rho)^3(g_\mu^2+g_{x_3}^2)}{(\mu+\rho)^2+x_3^2} \, d\mu d x_3 + \frac{3}{2} \int_{\mathbb T} \abs{x}^{-2} dx. 
\label{Energytwisttorus2}
\end{align}
Similar to what was described earlier in obtaining the second equality we have used the integral identity
\begin{align*}
\int_0^{2\pi}\int_0^{2\pi}\int_{a}^1 r(\rho + r\cos(\theta)) \, drd\theta d\phi = \int_0^{2\pi}\int_{\B_1\backslash \overline{\B}_a }(\mu+\rho) \, d\mu d x_3 d\phi .
\end{align*}

The Euler-Lagrange equation can be obtained in the standard way by taking variations $\varphi \in \mathbf{C}^\infty_c(\B_1\backslash \overline{\B}_a)$ where 
$\B_1\backslash \overline{\B}_a=\lbrace (\mu,x_3)\in \R^2 : a^2<\mu^2+x_3^2 <1 \rbrace$. This calculation again leads to the Euler-Lagrange equation given 
by \eqref{Eulertorustwist} where we assume
%
%
%
without loss of generality that the boundary condition on the rotation angle function $g$ is set to $g(\mu,x_3)=0$ for $(\mu,x_3)\in \partial \B_a$ 
and $g(\mu,x_3)=2\pi k$ for $(\mu,x_3)\in \partial \B_1$ with $k \in \Z$. Therefore solutions to \eqref{Eulertorustwist} satisfy, 
\begin{align}
0 = \int_{\B_1\backslash\overline{\B}_a} {\rm div} \frac{(\mu+\rho)^3 \nabla g}{(\mu+\rho)^2 + x_3^2} \, dx 
= \int_{\partial \B_1} \frac{(\mu+\rho)^3 \nabla g\cdot n}{(\mu+\rho)^2 + x_3^2} -  
\int_{\partial \B_a}\frac{(\mu+\rho)^3 \nabla g\cdot n}{(\mu+\rho)^2 + x_3^2}.
\end{align}
Subsequently 
\begin{align*}
\frac{\mathbb{F}[u; {\mathbb T}] - 3/2 \int_{\mathbb T} \abs{x}^{-2} dx}{\pi} =& \int_{\B_1\backslash \overline{\B}_a} \bigg[ \frac{(\mu+\rho)^3\abs{\nabla g}^2}{(\mu+\rho)^2 + x_3^2}  
+ g\,{\rm div}\left( \frac{(\mu+\rho)^3 \nabla g}{(\mu+\rho)^2 + x_3^2} \right) \bigg] \, dx \nonumber \\
=&  \int_{\partial \B_1} \frac{g(\mu+\rho)^3 \nabla g\cdot n}{(\mu+\rho)^2 + x_3^2} \,d{\mathcal H}^1
-  \int_{\partial \B_a} \frac{g(\mu+\rho)^3 \nabla g \cdot n}{(\mu+\rho)^2 + x_3^2} \, d{\mathcal H}^1.
\end{align*}
Hence taking into account the boundary conditions, e.g. $g=0$ on $\partial \B_a$ and $g=2\pi k$ on $\partial \B_1$, we gain that if $g$ is a solution of $(\ref{Eulertorustwist})$ then,
\begin{align}
\frac{\mathbb{F}[u; {\mathbb T}] - 3/2 \int_{\mathbb T} \abs{x}^{-2} dx}{\pi} &= 2\pi k \int_{\partial \B_1} \frac{(\mu+\rho)^3 \nabla g\cdot n}{(\mu+\rho)^2 + x_3^2} \,d\mathcal{H}^1, 
\qquad k \in \Z. \label{identity3}
\end{align}

Evidently $(\ref{Eulertorustwist})$ with the stated boundary conditions has a unique solution. Indeed if $\overline g, \underline g$ are two solutions to 
$(\ref{Eulertorustwist})$ with $g=0$ on $\partial \B_a$ and $g=2\pi k$ on $\partial \B_1$ 
then $g=\overline g-\underline g$ solves $(\ref{Eulertorustwist})$ with $g=0$ on $\partial [\B_1\backslash \overline{\B}_a]$. Then by $(\ref{identity3})$
\begin{align} 
\int_{ \B_1\backslash \overline{\B}_a} \frac{(\mu+\rho)^3 \abs{\nabla g}^2}{(\mu+\rho)^2 + x_3^2}  \, dx = 0.
\end{align}

However in view of $\rho>1$ this gives $\abs{\nabla g}^2 \equiv 0$ and so invoking the boundary conditions $g=0$, i.e., $\overline g= \underline g$. 
As existence follows from standard arguments it follows that (\ref{Eulertorustwist}) has a unique smooth solution $g=g(\mu, x_3; k)$ for each $k \in \Z$.

\section{Euler-Lagrange analysis and twists as classical solutions}
\setcounter {equation}{0}

The goal of this section is to examine the solution $g=g(\mu, x_3; k)$ to $(\ref{Eulertorustwist})$ with the prescribed boundary conditions in relation to the 
Euler-Lagrange system (\ref{EL}) associated with ${\mathbb F}$ on ${\mathcal A}_\phi({\mathbb T})$. To this end recall that the system takes the form 
\begin{align}
\frac{(\nabla u)^t}{\abs{u}^2} \left[ {\Delta u} + \frac{\abs{\nabla u}^2}{\abs{u}^2} u - \frac{2}{\abs{u}^2}  \nabla u (\nabla u)^tu \right] &= \nabla p.
\end{align}
For the ease of notation from now on we shall set $\xi = \mu + \rho$. Hence using the identities 
$(\ref{twistcalculations})$ we have  
\begin{align}
\begin{cases}
(\nabla u)(\nabla u)^t u = Qx + \langle \nabla g,x \rangle \dot{Q}x,\\
\abs{\nabla u}^2 \abs{u}^{-2}  = (3 + \xi^2 \abs{\nabla g}^2) \abs{x}^{-2}, \\
\Delta u = 2 \xi^{-1} g_\xi \dot{Q}x + \Delta g \dot{Q}x + \abs{\nabla g}^2 \ddot{Q}x. 
\end{cases} \label{twistcalculations2}
\end{align}
Therefore from $(\ref{twistcalculations2})$ and a basic calculation we obtain
\begin{align*}
{\Delta u} + \frac{\abs{\nabla u}^2}{\abs{u}^2} u - \frac{2}{\abs{u}^2}  \nabla u (\nabla u)^tu & = \left(\frac{2g_{\xi}}{\xi} + 
\Delta g -\frac{2 \langle \nabla g,x \rangle }{\abs{x}^2}\right)\dot{Q}x + \abs{\nabla g}^2 \ddot{Q}x \\
& + \frac{1 + \xi^2 \abs{\nabla g}^2}{\abs{x}^2} Qx. 
\end{align*}

Now since we have $\Delta g = \Delta_{\xi,x} g  + g_{\xi}/\xi$ where the $\Delta_{\xi,x}$ denotes the Laplacian with respect to the $\xi$ and $x_3$ 
variables we can rewrite this as 
\begin{align}
{\Delta u} + \frac{\abs{\nabla u}^2}{\abs{u}^2} u - \frac{2}{\abs{u}^2}  \nabla u (\nabla u)^tu =& \left(\frac{3g_{\xi}}{\xi} + 
\Delta_{\xi,x} g - \frac{2 \langle \nabla g,x \rangle }{\abs{x}^2}\right)\dot{Q}x + \abs{\nabla g}^2 \ddot{Q}x \nonumber \\
& + \frac{1 + \xi^2 \abs{\nabla g}^2}{\abs{x}^2} Qx. 
\end{align}
Now upon recalling that the desired twist solution satisfies $(\ref{Eulertorustwist})$ we have that
\begin{align*}
{\rm div} \left( \frac{\xi^3 \nabla g}{\xi^2 + x_3^2} \right)  = \frac{\xi^3 \Delta_{\xi,x} g}{\xi^2 + x_3^2} + 
\left( \frac{3\xi^2}{\xi^2 + x_3^2}- \frac{2\xi^4}{(\xi^2+x_3^2)^2}\right)g_{\xi} - \frac{2\xi^3 x_3 g_{x_3}}{(\xi^2 + x_3^2)^2} =0. 
\end{align*}

Thus dividing both sides by $\xi^3/(\xi^2+x_3^2)$ and taking the negative terms to one side gives 
\begin{equation*}
\Delta_{\xi,x}g + \frac{3 g_{\xi}}{\xi} = 2\left( \frac{\xi g_{\xi} + x_3 g_{x_3} }{\abs{x}^2} \right) = 2\frac{\langle \nabla_{\xi,x} g, z\rangle }{\abs{x}^2},
\end{equation*} 
where $z=(\xi,x_3)^t$ and $\nabla_{\xi,x}$ denotes the gradient with respect to the $(\xi, x_3)$ variable. Now since 
$\langle \nabla g, x \rangle = \langle \nabla_{\xi,x}g, z \rangle$ we obtain,
\begin{align*}
\Delta_{\xi,x}g + \frac{3 g_{\xi}}{\xi}  = 2\frac{\langle \nabla g, x\rangle }{\abs{x}^2},
\end{align*}
and so as a result 
\begin{align}
{\Delta u} + \frac{\abs{\nabla u}^2}{\abs{u}^2} u - \frac{2}{\abs{u}^2}  \nabla u (\nabla u)^tu & =  \abs{\nabla g}^2 \ddot{Q}x + 
\frac{1 + \xi^2 \abs{\nabla g}^2}{\abs{x}^2} Qx. 
\end{align}

Next referring to the definition of $Q$ basic calculation gives $\dot{Q} = J_1 Q $ and $\ddot{Q}=-J_2Q$, where
\begin{align}
J_1 = \begin{bmatrix}
0 & -1 & 0 \\
1 & 0 & 0 \\
0 & 0 & 0
\end{bmatrix}, \quad
 J_2 = \begin{bmatrix}
 1 & 0 & 0 \\
 0 & 1 & 0 \\
 0 & 0 & 0
 \end{bmatrix}. 
\end{align}

Hence with the above notation the Euler-Lagrange associated with the twist $u$, satisfying \eqref{Eulertorustwist}, simplifies to 
\begin{align}
\frac{(\nabla u)^t}{\abs{x}^4} (I + \xi^2 \abs{\nabla g}^2 I - \abs{x}^2\abs{\nabla g}^2 J_2) Qx 
& = (I + \xi^2 \abs{\nabla g}^2I - \abs{x}^2\abs{\nabla g}^2 J_2) \frac{x}{\abs{x}^4} \nonumber \\
& = \frac{1}{\abs{x}^4}\begin{bmatrix}
(1-x_3^2\abs{\nabla g}^2 ) x_1 \\
(1-x_3^2\abs{\nabla g}^2) x_2 \\
(1+\xi^2 \abs{\nabla g}^2) x_3 
\end{bmatrix} \nonumber \\
&= \nabla \left( -\frac{1}{2\abs{x}^2}  \right) + \frac{\abs{\nabla g}^2}{\abs{x}^4}\begin{bmatrix}
-x_3^2 x_1 \\
-x_3^2 x_2 \\
\xi^2 x_3 
\end{bmatrix} =\nabla p.
\end{align}

Considering the last line in the above equation it is plain that for $u$ to grant solution to the Euler-Lagrange equation it must be that  
\begin{align}
-\frac{1}{2} \nabla |x|^{-2} + \frac{\abs{\nabla g}^2}{\abs{x}^4} (-x_3^2 x_1, -x_3^2 x_2, \xi^2 x_3)^t =\nabla p,
\label{grad1}
\end{align}
or equivalently that the second term on the left is a gradient. But for this to be the case the latter term must necessarily be curl-free 
and so this leads to the system of equations
\begin{align}
0 & = \frac{\partial }{\partial x_1} \left( -\frac{x_3^2 \abs{\nabla g}^2}{\abs{x}^4}x_2 \right) 
- \frac{\partial }{\partial x_2} \left( -\frac{x_3^2 \abs{\nabla g}^2}{\abs{x}^4}x_1 \right), \label{eq1}\\
0 & = \frac{\partial }{\partial x_1} \left( \frac{\xi^2 \abs{\nabla g}^2}{\abs{x}^4}x_3 \right) 
- \frac{\partial }{\partial x_3} \left( -\frac{x_3^2 \abs{\nabla g}^2}{\abs{x}^4}x_1 \right),\label{eq2}\\
0 & = \frac{\partial }{\partial x_2} \left( \frac{\xi^2 \abs{\nabla g}^2}{\abs{x}^4}x_3 \right) 
- \frac{\partial }{\partial x_3} \left( -\frac{x_3^2 \abs{\nabla g}^2}{\abs{x}^4}x_2 \right). \label{eq3}
\end{align}

It can be easily verified that equation $(\ref{eq1})$ is satisfied for any twist map since here we have 
\begin{align*}
&  \frac{\partial }{\partial x_1} \left( -\frac{x_3^2 \abs{\nabla g}^2}{\abs{x}^4}x_2 \right) 
- \frac{\partial }{\partial x_2} \left( -\frac{x_3^2 \abs{\nabla g}^2}{\abs{x}^4}x_1 \right) \\
&= x_3^2 \abs{\nabla g}^2\left[ \frac{\partial }{\partial x_2} \frac{x_1}{\abs{x}^4} 
- \frac{\partial}{\partial x_1} \frac{x_2}{\abs{x}^4} \right] +  \frac{x_3^2}{\abs{x}^4} 
\left[ x_1\frac{\partial \abs{\nabla g}^2}{\partial x_2} -  x_2\frac{\partial \abs{\nabla g}^2}{\partial x_1}  \right] \\
&= \frac{x_3^2}{\abs{x}^4} \left[ x_1 \frac{x_2}{\xi}\frac{\partial \abs{\nabla g}^2}{\partial \xi} 
-  x_2\frac{x_1}{\xi}\frac{\partial \abs{\nabla g}^2}{\partial \xi}  \right]=0.
\end{align*}
We point out that the last line results upon noting the relations 
\begin{equation}
\frac{\partial }{\partial x_1} = \frac{x_1}{\xi} \frac{\partial}{\partial \xi}-\frac{x_2}{\xi^2}\frac{\partial}{\partial \phi}, \qquad 
\frac{\partial }{\partial x_2} = \frac{x_2}{\xi} \frac{\partial}{\partial \xi}+\frac{x_1}{\xi^2}\frac{\partial}{\partial \phi}.
\end{equation} 
Using this we can again see that we can write $(\ref{eq2})$ and $(\ref{eq3})$ as a single equation in the following way,
\begin{align}
\frac{\partial }{\partial x_1} \left( \frac{\xi^2 \abs{\nabla g}^2}{\abs{x}^4}x_3 \right) 
&+ \frac{\partial }{\partial x_3} \left( \frac{x_3^2 \abs{\nabla g}^2}{\abs{x}^4}x_1 \right) = \nonumber \\
& = \frac{x_1 x_3}{\xi} \frac{\partial}{\partial \xi} \left( \frac{\xi^2 \abs{\nabla g}^2}{\abs{x}^4} \right) 
+ x_1\frac{\partial }{\partial x_3} \left( \frac{x_3^2 \abs{\nabla g}^2}{\abs{x}^4} \right),\\
\frac{\partial }{\partial x_2} \left( \frac{\xi^2 \abs{\nabla g}^2}{\abs{x}^4}x_3 \right) 
&- \frac{\partial }{\partial x_3} \left( -\frac{x_3^2 \abs{\nabla g}^2}{\abs{x}^4}x_2 \right) = \nonumber \\
& = \frac{x_2 x_3}{\xi} \frac{\partial}{\partial \xi} \left( \frac{\xi^2 \abs{\nabla g}^2}{\abs{x}^4} \right) 
+ x_2\frac{\partial }{\partial x_3} \left( \frac{x_3^2 \abs{\nabla g}^2}{\abs{x}^4} \right).
\end{align}
Therefore it is apparent that $(\ref{eq2})$ and $(\ref{eq3})$ become
\begin{align}
x_3 \frac{\partial}{\partial \xi} \left( \frac{\xi^2 \abs{\nabla g}^2}{\abs{x}^4} \right) + \xi\frac{\partial }{\partial x_3} 
\left( \frac{x_3^2 \abs{\nabla g}^2}{\abs{x}^4} \right)=0. \label{eq4}
\end{align}
Now since we have the identities
\begin{align}
\frac{\partial }{\partial \xi} \left( \frac{\xi^2}{\abs{x}^4} \right) = \frac{2\xi (x_3^2 - \xi^2)}{(\xi^2+x_3^2)^3} 
=- \frac{\xi}{x_3} \frac{\partial }{\partial x_3} \left( \frac{x_3^2}{\abs{x}^4} \right),
\end{align}
we obtain that $(\ref{eq4})$ simplifies further to 
\begin{align}
\frac{x_3\xi}{\abs{x}^4} & \left[ \xi \frac{\partial \abs{\nabla g}^2}{\partial \xi} 
x_3\frac{\partial \abs{\nabla g}^2 }{\partial x_3} \right] = 0 \implies
 \xi \frac{\partial \abs{\nabla g}^2}{\partial \xi}  + x_3\frac{\partial \abs{\nabla g}^2 }{\partial x_3} = 0.
\end{align}

Hence for a solution $g$ to $(\ref{Eulertorustwist})$ with the prescribed boundary conditions to furnish a solution 
to the Euler-Lagrange system (\ref{EL}) associated with ${\mathbb F}$ it is necessary for $g$ to satisfy
\begin{align}
\xi \frac{\partial \abs{\nabla g}^2}{\partial \xi}  + x_3\frac{\partial \abs{\nabla g}^2 }{\partial x_3} = 0. \label{gradcond}
\end{align}

We now show that (\ref{gradcond}) is also sufficient. Indeed assuming (\ref{gradcond}) the desired conclusion will 
follow upon showing that $(\ref{grad1})$ holds. Towards this end set $f$ to be the function,
\begin{align}
f(\xi,x_3) = -\int_0^{\xi} \frac{x_3^2\abs{\nabla g}^2}{(\tau^2 +  x_3^2)^2}\tau d\tau.
\end{align}
Then one can easily verify that,
\begin{align}
\frac{\partial f}{\partial x_1} = -\frac{x_3^2 \abs{\nabla g}^2}{\abs{x}^4}x_1, \quad \frac{\partial f}{\partial x_2} 
= -\frac{x_3^2 \abs{\nabla g}^2}{\abs{x}^4}x_2.
\end{align}
Furthermore using $(\ref{eq4})$ it is plain that 
\begin{align}
\frac{\partial f}{\partial x_3} &= - \int_0^{\xi} \frac{\partial }{\partial x_3} 
\frac{x_3^2\abs{\nabla g}^2 \tau}{(\tau^2 +  x_3^2)^2}\, d\tau \nonumber \\
& = \int_0^{\xi} \frac{\partial }{\partial \tau} \frac{\tau^2\abs{\nabla g}^2 x_3}{(\tau^2 +  x_3^2)^2} \, d\tau 
= \frac{\xi^2\abs{\nabla g}^2}{\abs{x}^4} x_3.
\end{align}
As a result $\nabla f = \abs{\nabla g}^2 \abs{x}^{-4} (-x_3^2 x_1, -x_3^2 x_2, \xi^2 x_3)^t$.

\begin{thm} A twist map $u$ with the corresponding angle of rotation function $g=g(\mu, x_3; k)$ satisfying $(\ref{Eulertorustwist})$ 
and $g=0$ on $\partial \B_a$, $g=2\pi k$ on $\partial \B_1$ $($with $k \in \Z$$)$  is a solution to the Euler-Lagrange system 
$(\ref{EL})$ associated with ${\mathbb F}$ on ${\mathcal A}_\phi({\mathbb T})$ iff it satisfies $(\ref{gradcond})$.
\end{thm}

\noindent{ $\dag$ {\scriptsize DEPARTMENT OF MATHEMATICS, UNIVERSITY 
OF SUSSEX, FALMER, BRIGHTON BN1 9RF, ENGLAND, UK.}}    
\\[2 mm]
\noindent{\textit{{\small E-mail address:}} \textrm{{\small 
a.taheri@sussex.ac.uk}}}


\begin{thebibliography}{999}

\bibitem{AIM} K.~Astala, T.~Iwaniec, G.~Martin, {\it Elliptic Partial Differential Equations and Quasiconformal Mappings in the Plane}, 
Princeton Mathematical Series, Vol.~{\bf 48}, Princeton University Press, 2009.

\bibitem{AIMO} K.~Astala, T.~Iwaniec, G.~Martin, J.~Onninen, Extremal Mappings of Finite Distortion, {\it Proc. Lond. Math. Soc.}, 
Vol.~{\bf 91}, 2005, pp.~655-702.   
 
\bibitem{Brothers} J.E.~Brothers, W.P.~Ziemer, Minimal rearrangements of Sobolev functions, 
{\it Acta Univ. Carolin. Math. Phys.}, Vol.~{\bf 28}, 1987, pp 13€-24.

  
\bibitem{Federer} H.~Federer, {\it Geometric Measure Theory}, Classics in Mathematics, 
Vol.~{\bf 153}, Springer-Verlag, 1969. 

\bibitem{HM-C} S.~Hencl, C.~Mora-Corral, Diffeomorphic approximation of continuous almost everywhere injective Sobolev 
deformations in the plane, {\it Q. J. Math.}, Vol.~{\bf 66}, 2015, pp.~1055-1062.

\bibitem{IO} T.~Iwaniec, J.~Onninen, $n$-harmonic mappings between annuli: the art of integrating free Lagrangians, 
{\it Mem. Amer. Math. Soc.}, Vol.~{\bf 218}, viii+105 pp., 2012.
 
\bibitem{IS} T.~Iwaniec, V.~Sverak, Mappings with integrable dilatations, {\it Proc. Amer. Math. Soc.}, Vol.~{\bf 118}, 
1993, pp.~181-188.
 
\bibitem{MSZ} J.~Mal{\'y}, D.~Swanson, W.~Ziemer, The co-area formula for Sobolev mappings, {\it Trans. Amer. Math. Soc.}, 
Vol.~{\bf 355}, 2003, pp.~477-492.
 
\bibitem{MC} C.B.~Morrey. {\it Multiple integrals in the calculus of variations}, Classics in Mathematics, Vol.~ {\bf 130}, 
Springer, 1966.

\bibitem{CT} C.~ Morris, A.~Taheri, On the Uniqueness of Energy Minimisers in Homotopy Classes, Submitted for publications, 2017.

\bibitem{MST} S.~M\"uller, S.J.~Spector, Q. Tang, Invertibility and a topological property of Sobolev maps, {\it SIAM J. Math. Anal.}, 
Vol.~{\bf 27}, pp.~959-976, 1996.

\bibitem{ShT} M.S.~Shahrokhi-Dehkordi, A.Taheri,  Generalised twists, stationary loops and the Dirichlet energy over a space of 
measure preserving maps, {\it Calc. Var. $\&$ PDEs}, Vol.~{\bf 35}, 2009, pp.~191-213.

\bibitem{ShT2} M.S.~Shahrokhi-Dehkordi, A.Taheri,  Generalised twists, ${\bf SO}(n)$ and the $p$-energy over a space of measure 
preserving maps, {\it Ann. Inst. Henri Poincar\`e, Analyse non lineaire}, Vol.~{\bf 26}, 2009, pp.~1897-1924.



\bibitem{S} V.~Sverak, Regularity properties of deformations with finite energy, {\it Arch. Rational Mech. Anal.}, Vol.~{\bf 100}, 
1988, pp.~105-127. 

\bibitem{TA2} A.~Taheri, Local minimizers and quasiconvexity - the impact of Topology, 
{\it Arch. Rational Mech. Anal.}, Vol.~{\bf 176}, No.~{\bf 3}, 2005, pp.~363-414.

\bibitem{TA3} A.~Taheri, Minimizing the Dirichlet energy over a space of measure preserving maps, 
{\it Top. Meth. Nonlinear Anal.}, Vol.~{\bf 33}, 2009, pp.~179-204.

\bibitem{TA4} A.~Taheri, Homotopy classes of self-maps of annuli, generalised twists and spin degree, 
{\it Arch. Rational Mech. Anal.}, Vol.~{\bf 197}, 2010, pp.~239-270.

\bibitem{TA5} A.~Taheri, Spherical twists, stationary loops and harmonic maps from generalised 
annuli into spheres, {\it NoDEA}, Vol.~{\bf 19}, 2012, pp.~79-95.

\bibitem{VG} S.K.~Vodopyanov, V.M.~Gol'dshtein. Quasiconformal mappings and spaces of functions with 
generalized first derivatives, {\it Siberian Math. J.}, Vol.~{\bf 17}, 1977, pp.~515-531.

\end{thebibliography}
\end{document}